\documentclass[12pt]{article}
\usepackage{mathrsfs}
\usepackage{amssymb}
\usepackage[hang, flushmargin]{footmisc}
\usepackage{multirow}
\usepackage{amsfonts}
\usepackage{graphicx}
\usepackage{subfigure}
\usepackage{amsmath, amsthm, amssymb, latexsym}
\usepackage[colorlinks,  linkcolor=blue,  anchorcolor=gree,  citecolor=red]{hyperref}
\usepackage[numbers, sort&compress]{natbib}
\usepackage{xcolor}
\usepackage{enumerate}
\usepackage{bm}
\usepackage{indentfirst}
\makeatletter

\newcommand{\Rmnum}[1]
{\expandafter\@slowromancap\romannumeral #1@}
\usepackage{geometry}
\makeatother

\newtheorem{thm}{Theorem}[section]

\newtheorem{prop}[thm]{Proposition}
\newtheorem{lemma}[thm]{Lemma}
\newcounter{foo}[subsection]

\newtheorem{claim}[foo]{Claim}

\newtheorem{example}[thm]{Example}
\newtheorem{defin}[thm]{Definition}

\theoremstyle{definition}
\newtheorem*{remark}{Remark}
\newenvironment{re}{\begin{remark} \rm}{\end{remark}}

\usepackage{CJK}
\begin{document}		
	\renewcommand{\baselinestretch}{1.3}
	\title{The maximum sum of sizes of non-empty cross $t$-intersecting families}
	\author{Shuang Li\textsuperscript{a}\quad Dehai Liu\textsuperscript{a,$\ast$}\quad Deping Song\textsuperscript{a} \quad Tian Yao\textsuperscript{b}
		\\
		\\{\footnotesize  \textsuperscript{a} \em  Laboratory of Mathematics and Complex Systems (MOE),}  \\
		{\footnotesize  \em School of Mathematical Sciences, Beijing Normal University, Beijing, 100875, China }
		\\{\footnotesize  \textsuperscript{b} \em School of Mathematical Sciences, Henan Institute of Science and Technology, Xinxiang 453003, China}}
	\date{}
	\maketitle
	
	\footnote{\quad * Corresponding author.}
	\footnote{\scriptsize\qquad{\em E-mail address:} ~~lishuangyx@mail.bnu.edu.cn (S. Li),  liudehai@mail.bnu.edu.cn (D. Liu), songdeping@mail.bnu.edu.cn(D. Song), tyao@hist.edu.cn(T. Yao).}
	\begin{abstract}
		Let $[n]:=\lbrace 1,2,\ldots,n  \rbrace$, and $M$ be a set of positive integers. Denote the family of all subsets of $[n]$ with sizes in $M$ by $\binom{\left[n\right]}{M}$. The non-empty families  $\mathcal{A}\subseteq\binom{\left[n\right]}{R}$ and   $\mathcal{B}\subseteq \binom{\left[n\right]}{S}$ are said to be cross $t$-intersecting if $|A\cap B|\geq t$ for all $A\in \mathcal{A}$ and $B\in \mathcal{B}$.  In this paper, we determine the maximum sum of sizes of non-empty cross $t$-intersecting families, and characterize the extremal families.  Similar result for finite vector spaces is also proved.
		
	\end{abstract}

	\section{Introduction}
	
	Let $\left[n\right]:=\left\{ 1, 2, \ldots, n\right\}$, and
	$\binom{\left[n\right]}{k}$ be the family of all $k$-subset of $[n]$. For a positive integer $t$, a family $\mathcal{F}\subseteq \binom{[n]}{k}$ is called \textit{$t$-intersecting} if $|A\cap B|\geq t$   for any $A, B\in \mathcal{F}$. The structures of $t$-intersecting families of $\binom{[n]}{k}$ with maximum size have been completely determined \cite{MR0140419,MR0519277,MR0771733,MR1429238}, which are known as the Erd\H{o}s-Ko-Rado Theorem for sets.
	
     As a generalization, cross $t$-intersecting families can be considered.	For  $M\subseteq [n]$, write $\binom{\left[n\right]}{M}:=\bigcup_{m\in M} \binom{\left[n\right]}{m}$.  Let  $R$ and $S$ be subsets of $[n]$ and $t$ a positive integer with $ t\leq \min R\cup S$.  We say that two non-empty families  $\mathcal{A}\subseteq\binom{\left[n\right]}{R}$ and   $\mathcal{B}\subseteq \binom{\left[n\right]}{S}$ are \textit{non-empty cross $t$-intersecting} if $|A\cap B|\geq t$ for all $A\in \mathcal{A}$ and $B\in \mathcal{B}$.    The problem of maximizing the sum of  sizes of non-empty cross $t$-intersecting families has been attracting much attention. In \cite{MR0219428}, Hilton and Milner solved this problem for  $R=S=\lbrace  r\rbrace$ and  $t=1$. Frankl and Tokushige \cite{MR1178386} obtained the result for  $R=\lbrace r\rbrace$, $S=\lbrace s\rbrace$ and $t=1$. 
	Wang and Zhang \cite{MR2971702} solved this problem for $R=\lbrace r\rbrace$, $S=\lbrace s\rbrace$ and   $t<\min \lbrace r,s \rbrace$, determining structures of extremal families. 	
	
	Recently, Borg and Feghali \cite{MR4441255} got the maximum sum of  sizes for  $R=[r]$, $S=[s]$ and $t=1$. There are some results on weighted version of this problem, see \cite{MR4644282,Liu202330910,MR4618236} for details. 
	For general $R$, $S$ and $t$, by  \cite[Theorem 1.2]{MR4618236}, we can derive the maximum sum of sizes  if $3\max R\cup S-t\leq n$.  In this paper, by characterizing a special independent sets of non-complete bipartite graph, we prove the same result for $\max R+\max S -t<n$.  Moreover, when $n$ is sufficiently large, the extermal structures are characterized. The following theorem is one of our main results.

	\begin{thm}\label{1}
	Let $n$, $r$, $s$ and $t$ be positive integers, and $R$, $S$ be  subsets of $[n]$ with  $t\leq \min R\cup S$,  $r=\max R$,  $s=\max S$ and  $r+s-t<n$.  
		If $\mathcal{A}\subseteq\binom{\left[n\right]}{R}$ and   $\mathcal{B}\subseteq \binom{\left[n\right]}{S}$ are non-empty  cross $t$-intersecting,   then 	  
		\begin{equation}\tag{$1.1$}
			\left|\mathcal{A}\right|+\left|\mathcal{B}\right|\leq \max\left\{ 1+ \sum_{i\in S}\sum_{j=t}^{i}\binom{r}{j}\binom{n-r}{i-j},\  1+\sum_{i\in R}\sum_{j=t}^{i}\binom{s}{j}\binom{n-s}{i-j}  \right\}.
		\end{equation}
		Moreover, the following hold.
		\begin{enumerate}[\normalfont(1)]
			\vspace{-0.2cm}	\item If $R=S$ and $2r< n$,  then  equality holds if and only if one of the following holds:
			\begin{enumerate}[\normalfont(i)]		
				\item $\mathcal{A}=\left\{ A \right\}$ and $\mathcal{B}=\left\{ B\in \binom{[n]}{S}: |B\cap A|\geq t\right\}$ for some $A\in \binom{[n]}{r}$;
				\item $\mathcal{B}=\left\{ B \right\}$ and $\mathcal{A}=\left\{ A\in \binom{[n]}{R}: |A\cap B|\geq t\right\}$ for some $B\in \binom{[n]}{s}$;
				\item $R=S=\left\{1,  2 \right\}$ and $\mathcal{A}=\mathcal{B}=\left\{  C\in \binom{[n]}{[2]}: i\in C\right\}$ for some $i \in [n]$;
				\item $R=S=\left\{ 2 \right\}$ and $\mathcal{A}=\mathcal{B}=\left\{  C\in \binom{[n]}{ 2}: i\in C\right\}$ for some $i\in [n]$.			
			\end{enumerate}
			\vspace{-0.2cm}	\item If $R\neq S$ and 
			$\max \left\{ 2^{s+1}|R|(s-t)+2r+1, \ 2^{r+1}|S|(r-t)+2s+1 \right\}\leq n,$
			then equality holds if and only if one of the following holds:		
			\begin{enumerate}[\normalfont(i)]
				\vspace{-0.2cm}	\item $\max R\bigtriangleup S\in S$ and $\mathcal{A}=\left\{ A \right\}$ and $\mathcal{B}=\left\{ B\in \binom{[n]}{S}: |B\cap A|\geq t\right\}$ for some $A\in \binom{[n]}{r}$;
				\vspace{-0.5cm}	\item $\max R\bigtriangleup S \in R$ and $\mathcal{B}=\left\{ B \right\}$ and $\mathcal{A}=\left\{ A\in \binom{[n]}{R}: |A\cap B|\geq t\right\}$ for some $B\in \binom{[n]}{s}$.
			\end{enumerate}
		\end{enumerate}
	\end{thm}
	\begin{re}
		The condition $r+s-t<n$ in Theorem \ref{1} is necessary.    Let $R=S=\left\{1,   z \right\}$,   $t
		=1$ and $n=2z-1$,   where $z\geq 3$.   It is easily check that   $\mathcal{A}=\mathcal{B}=\binom{[n]}{z}$ are cross $1$-intersecting families with
		$ |\mathcal{A}|+|\mathcal{B}|= 2\binom{n}{z}>1+z+\binom{n}{z}. $ 
	\end{re}  
	
	Intersection problems have been extended to some other mathematical objects, for example, vector spaces.  Let $V$ be an  $n$-dimensional vector space over a finite filed $\mathbb{F}_{q}$, where $q$ is a prime power.  
	Denote the family of all $k$-dimensional subspaces of $V$ by ${V \brack k}$.  Then the size of ${V \brack k}$ equals ${n \brack k}:=\prod_{i=0}^{k-1}\frac{q^{n-i}-1}{q^{k-i}-1}$.  The Erd\H{o}s-Ko-Rado Theorem for vector spaces was proved in \cite{MR0721612,MR0810699,MR0867648,MR0382015,MR0382016,MR2231096}, in this paper, we focus on cross $t$-intersecting families for  vector spaces.  
	For $M\subseteq [n]$, write ${V \brack M}:=\bigcup_{m\in M} {V \brack m}$.   Let  $R$ and $S$ be subsets of $[n]$ and $t$ a positive integer with $ t\leq \min R\cup S$.  We say that two non-empty families  $\mathcal{A}\subseteq{ V \brack R}$ and   $\mathcal{B}\subseteq {V \brack S}$ are \textit{non-empty cross $t$-intersecting} if $\dim(A\cap B)\geq t$ for all $A\in \mathcal{A}$ and $B\in \mathcal{B}$.
	Wang and  Zhang \cite{MR2971702}  determined the maximum sum of sizes and the extremal families of non-empty cross $t$-intersecting families for  $R=\lbrace r\rbrace$, $S=\lbrace s\rbrace$ and  $t<\min \lbrace r,s \rbrace$.

	In this paper, we  determine the maximum sum of sizes of non-empty cross $t$-intersecting families for general $R$, $S$ and $t$, and characterize the extremal families under certain conditions. 
	\begin{thm}\label{2}
		Let $n$, $r$, $s$ and $t$ be positive integers, and $R$, $S$ be  subsets of $[n]$ with  $t\leq \min R\cup S$,  $r=\max R$,  $s=\max S$ and  $r+s-t<n$.  
		If $\mathcal{A}\subseteq{V \brack R}$ and  $\mathcal{B}\subseteq {V \brack S}$ are non-empty  cross $t$-intersecting,   then		
		\begin{equation} \tag{$1.2$}
			\left|\mathcal{A}\right|\!+\! \left|\mathcal{B}\right|\!\leq \!\max\left\{ 1+ \sum_{i\in S}\sum_{j=t}^{i}q^{(r-j)(i-j)}{r \brack j}{n-r \brack i-j},\ 
			1+\sum_{i\in R}\sum_{j=t}^{i}q^{(s-j)(i-j)}{s \brack j}{n-s \brack i-j} \right\}.
		\end{equation}
		Moreover, the following hold. 
		\begin{enumerate}[\normalfont(1)]
			\vspace{-0.2cm}	\item If $R=S$,  then  equality holds if and only if one of the following holds:
			\begin{enumerate}[\normalfont(i)]
				\vspace{-0.2cm}	\item  $\mathcal{A}=\left\{ A \right\}$ and $\mathcal{B}=\left\{ B\in {V \brack S} : \dim(B\cap A)\geq t\right\}$ for some $A\in {V \brack r}$;
					\item $\mathcal{B}=\left\{ B \right\}$ and $\mathcal{A}=\left\{ A\in {V \brack R}: \dim(A\cap B)\geq t\right\}$ for some $B\in {V \brack s}$. 
			\end{enumerate}		
			\vspace{-0.2cm}	\item    If $R\neq S$ and $s^{2}+r^{2}+sr\leq n$,    then  equality holds if and only if one of the following holds:
			\begin{enumerate}[\normalfont(i)]
				\vspace{-0.2cm}	\item  $\max R\bigtriangleup S \!\in \!S$ and $\mathcal{A}\!=\!\left\{ A \right\}$ and $\mathcal{B}=\left\{ B\in {V \brack S} : \dim(B\cap A)\geq t\right\}$ for some $A\in {V \brack r}$;
				\vspace{-0.6cm}	\item $\max R\bigtriangleup S \!\in\! R$ and $\mathcal{B}\!=\!\left\{ B \right\}$ and $\mathcal{A}=\left\{ A\in {V \brack R}: \dim(A\cap B)\geq t\right\}$ for some $B\in {V \brack s}$. 
			\end{enumerate}
		\end{enumerate}
	\end{thm}

	\section{Nontrivial independent sets of bipartite graphs}	
	Some concepts about bipartite graphs were proposed by Wang and Zhang in \cite{MR2971702}. We offer a concise explanation here.  
	Let $G(X,Y)$ be a non-complete bipartite graph with  bipartite sets $X$ and $Y$, the following concepts are applicable:
	
	\begin{itemize}  
		\item An independent set $I$ of $G(X,Y)$ is said to be \textit{nontrivial} if $I\cap X\neq \emptyset$ and $I\cap Y\neq \emptyset$.  
		
		\item By $\alpha(X,Y)$ and $\mathcal{I}(X,Y)$  denote the size and the set of maximum-sized nontrivial independent sets of $G(X,Y)$, respectively.  
		
		\item Set $\mathcal{F}(X)=\left\{ I\cap X: I\in \mathcal{I} \right\}$, $\mathcal{F}(Y)=\left\{ I\cap Y: I\in \mathcal{I} \right\}$ and $\mathcal{F}(X,Y)=\mathcal{F}(X)\cup \mathcal{F}(Y)$.  Each element of $\mathcal{F}(X,Y)$ is called a \textit{fragment} of $G(X,Y)$.
	\end{itemize}  
	
	Let $I\in \mathcal{I}(X,Y)$.  Then it is easy to see that  $I\cap X=X\backslash N(I\cap Y)$ and $I\cap Y=Y\backslash N(I\cap X)$, where $N(I\cap Y)$ and $N(I\cap X)$ are the neighborhoods of $I\cap Y$ and $I\cap X$, respectively. That is to say, in order to determine $\mathcal{I}(X,Y)$, it is sufficient to determine $\mathcal{F}(X)$ or $\mathcal{F}(Y)$. Moreover, we have $$\alpha(X,Y)=|I\cap X|+|Y\backslash N(I\cap X)|=|Y|-(|N(I\cap X)|-|I\cap X|).$$  With the notation $\varepsilon(X)=\min\left\{ |N(A)|-|A|: A\neq \emptyset, N(A)\neq Y\right\}$,   we have $\alpha(X,Y)=|Y|-\varepsilon(X)$. Similarly, write $\varepsilon(Y)=\min\lbrace |N(B)|-|B|: B\neq \emptyset, N(B)\neq X\rbrace.$ We have $\alpha(X,Y)=|X|-\varepsilon(Y)$. Furthermore, the elements in $\mathcal{F}(X,Y)$ have the following relationship.
		
	\begin{lemma}\textnormal{(\cite[Lemma 2.1]{MR2971702})} \label{3}
		Let $G(X,  Y)$ be a non-complete bipartite graph. Then $|Y|-\varepsilon(X)=|X|-\varepsilon(Y),$ and
		\begin{enumerate}[\normalfont(i)]
			\vspace{-0.2cm}	\item $A\in \mathcal{F}(X)$ if and only if $Y\backslash N(A)\in \mathcal{F}$,  and $N(Y\backslash N(A))=X\backslash A$;
			\vspace{-0.2cm}	\item $A\cap B$ and $A\cup B$ are both in $\mathcal{F}(X)$ if $A, B\in \mathcal{F}(X),   A\cap B\neq \emptyset$ and $N(A\cup B)\neq Y$.
		\end{enumerate}		
	\end{lemma}

	From the first statement of this lemma it follows that there exists a bijection $\phi:\mathcal{F}(X, Y)\rightarrow \mathcal{F}(X, Y)$ such that
	\begin{equation}
		\phi(A)=\left\lbrace
		\begin{array}{lllllllll}
			Y\backslash N(A)&&&\rm if~\it A\in \mathcal{F}(X);\\
			X\backslash N(A)&&&\rm if~\it A\in \mathcal{F}(Y) .
		\end{array}
		\nonumber
		\right.
	\end{equation}
	Moreover, we have $\phi^{-1}=\phi$.

	Let $X$ be a finite set and $\Gamma$  a group transitively acting on $X$. We say that the action of $\Gamma$ on $X$ is \textit{primitive} if $\Gamma$ preserves no nontrivial partition of $X$ and \textit{imprimitive} otherwise.   It is easy to see that if the action of $\Gamma$ on $X$ is transitive and imprimitive,  then there is a subset $B$ of $X$ such that $1<\left| B\right| < \left| X\right| $ and $\gamma(B)\cap B=B$ or $\emptyset$ for any $\gamma \in \Gamma$. In this case,  $B$ is called an \textit{imprimitive set} in $X$.
	
	Let $G(X,  Y)$ be a non-compete bipartite graph and  $\Gamma$ an automorphism group of $G(X,  Y)$ stabilizing $X$. If $O$ is an orbit under the action of $\Gamma$, then each vertex in  $O$  has the same degree, written as $d(O)$.  For an orbit $O$ under the action of $\Gamma$, we can restrict the action of $\Gamma $ on $O$,  denoted by $\Gamma\mid _{O}$.
	
	To deal with non-uniform cross $t$-intersecting families, we introduce the following theorem. 
	
	\begin{thm}\label{import}
	Let $G(X,  Y)$ be a non-complete bipartite graph and $\Gamma$  an automorphism group of $G(X,  Y)$ stabilizing $X$. Suppose that $\lbrace X_{i}:i\in [k] \rbrace$ and $\lbrace Y_{j}:j\in [\ell] \rbrace$ are all orbits in $X$ and $Y$ under the action of $\Gamma$, respectively, with  $d(X_{1})>d(X_{2})>\cdots>d(X_{k})$ and $d(Y_{1})>d(Y_{2})>\cdots>d(Y_{\ell})$. If for any non-singleton orbit $P$, $\Gamma\mid _{P}$ is primitive and $P$ is not contained in any nontrivial independent set, 
	then $$\alpha(X,  Y)=\max \left\{1+|Y|-d(X_{k}),  1+|X|-d(Y_{\ell})\right\}. $$
	Moreover,  one of the following holds:	
	\begin{enumerate}[\normalfont(i)]
		\vspace{-0.2cm}	\item$|X|-d(Y_{\ell})<|Y|-d(X_{k})$ and $\mathcal{F}(X)=\binom{X_{k}}{1}$;
		\vspace{-0.2cm}	\item$|X|-d(Y_{\ell})>|Y|-d(X_{k})$ and $\mathcal{F}(Y)=\binom{Y_{\ell}}{1}$;
		\vspace{-0.2cm}	\item$|X|-d(Y_{\ell})=|Y|-d(X_{k})$ and $\mathcal{F}(X,Y)\supseteq\binom{X_{k}}{1}  \cup\binom{Y_{\ell}}{1} \cup \phi \left(\binom{X_{k}}{1}\cup \binom{Y_{\ell}}{1}\right)$;
		\vspace{-0.2cm}	\item there exists a $F\in \mathcal{F}(X,Y)$  such that $|F|\geq 2$ and $|F\cap P|\leq 1$ for any $P$.
	\end{enumerate}	
\end{thm}

	To prove Theorem \ref{import}, we need the following three lemmas.
	The next lemma is a generalization of (\cite[Lemma 2.2]{MR2971702}), but the proof is identical. We include a concise proof here for completeness.
	\begin{lemma}\label{4}
		Let $G(X,  Y)$ be a non-compete bipartite graph and  $\Gamma$ an automorphism group of $G(X,  Y)$ stabilizing $X$. Suppose that $A\in\mathcal{F}(X,  Y)$ such that $\gamma(A)\cap A\neq \emptyset$ for some $\gamma \in \Gamma$.  If $|A|\leq |\phi(A)|$,   then $A\cup \gamma(A)$ and $A\cap \gamma(A)$ are both in $\mathcal{F}(X,  Y)$. 
	\end{lemma}
	\begin{proof}[\textnormal{\textbf{Proof.}}]
		Without loss of generality, suppose $A\in \mathcal{F}(X)$.  Since $\gamma$ is an automorphism and $\gamma(X)=X$, we have $\gamma(A)\subseteq X$,  $|\gamma(A)|=|A|$ and $|N(\gamma(A))|=|\gamma(N(A))|=|N(A)|$. Thus $\gamma(A)\in \mathcal{F}(X)$.   
		
		It is sufficient to prove $N(A\cup \gamma(A))\neq Y$ by  Lemma \ref{3}.  Note that  $N(A\cup\gamma(A))=N(A)\cup N(\gamma(A))$ and $N(A\cap \gamma(A))\subseteq N(A)\cap N(\gamma(A))$. Then
		\begin{equation*}  
			\begin{aligned}  
				|N(A\cup \gamma(A))| &= |N(A)\cup N(\gamma(A))| = |N(A)| + |N(\gamma(A))| - |N(A)\cap N(\gamma(A))| \\  
				&\leq 2|N(A)| - |N(A\cap\gamma(A))|  \leq 2|N(A)| - (|A\cap \gamma(A)| + \varepsilon(X)) \\  
				&= |N(A)| + |A| + \varepsilon(X) - (|A\cap \gamma(A)| + \varepsilon(X)) = |N(A)| + |A| - |A\cap \gamma(A)|\\& < |N(A)| + |A| \leq |N(A)| + |Y\backslash N(A)| = |Y|,    
			\end{aligned}  
		\end{equation*} 
		as desired. 
	\end{proof}

	\begin{lemma}\label{c}
		Let $G(X,  Y)$ be a non-compete bipartite graph and  $\Gamma$ an automorphism group of $G(X,  Y)$ stabilizing $X$. Suppose that, for any non-singleton orbit $P$ in $X$ under the action of $\Gamma$, $\Gamma\mid _{P}$ is primitive and $P$ is not contained in any nontrivial independent set. If   $|A|\leq |\phi(A)|$ for some  $A\in \mathcal{F}(X)$, then there exists a $F\in \mathcal{F}(X)$ such that $|F\cap O|\leq 1$ for any orbit $O$.
	\end{lemma}
	\begin{proof}[\textnormal{\textbf{Proof.}}]
		It is clear that $A\cap Y_{j}=\emptyset$ for any $j\in[\ell]$ since $A\subseteq X$.	 If $|A\cap X_{i}|\leq 1$ for any $i\in[k]$, then there is nothing to prove. In the following, we may assume  $|A\cap X_{p}|\geq 2$ for some $X_{p}$.
		
		Note that $X_{p}$ is not contained in $A$.  Then  $|A\cap X_{p}|<|X_{p}|$. Since $\Gamma\mid _{X_{p}}$ is primitive, there exists $\gamma\in \Gamma$ such that  $\emptyset \neq A\cap X_{p}\cap \gamma(A\cap X_{p})\neq A\cap X_{p}$, i.e.   $\emptyset \neq A\cap \gamma(A)\cap X_{p}\neq A\cap X_{p}$. Consequently  $\emptyset \neq A\cap \gamma(A)\neq A$.  By Lemma \ref{4},  we have $A\cap \gamma(A)\in \mathcal{F}(X)$. 
		Note that $|A\cap \gamma(A)|\leq |A|\leq |\phi(A)|\leq |\phi(A\cap \gamma(A))|$. Then we can use $A\cap \gamma(A)$ to replace $A$.
		
		 Repeating the above process, we finally get the desired  $F$.
	\end{proof}

	\begin{lemma}\label{b}
		Let $G(X,  Y)$ be a non-compete bipartite graph and  $\Gamma$ an automorphism group of $G(X,  Y)$ stabilizing $X$. Suppose that $\lbrace X_{i}:i\in [k] \rbrace$ and $\lbrace Y_{j}:j\in [\ell] \rbrace$ are all orbits in $X$ and $Y$ under the action of $\Gamma$, respectively, with  $d(X_{1})>d(X_{2})>\cdots>d(X_{k})$ and $d(Y_{1})>d(Y_{2})>\cdots>d(Y_{\ell})$. If there exists a $F\in \mathcal{F}(X,Y)$ such that $|F\cap O|\leq 1$ for any orbit $O$, then $$\alpha(X,Y)=\max \left\{1+|Y|-d(X_{k}), 1+|X|-d(Y_{\ell})\right\}.$$
	\end{lemma}
	\begin{proof}[\textnormal{\textbf{Proof.}}]
		In order to prove $\alpha(X,Y)=\max \left\{1+|Y|-d(X_{k}), 1+|X|-d(Y_{\ell})\right\}$, we just need to show $\binom{X_{k}}{1}\subseteq \mathcal{F}(X,Y)$ or $\binom{Y_{\ell}}{1}\subseteq \mathcal{F}(X,Y)$.
	 Without loss of generality, we may assume $F\in \mathcal{F}(X)$. 
	 
	 Set $F=\left\{  x_{i_{1}},  \cdots,  x_{i_{p}} \right\}$, where  $i_{1}< i_{2}< \cdots < i_{p}$, $x_{i_{1}}\in X_{i_{1}},  \cdots, x_{i_{p}}\in X_{i_{p}}$.  To prove $\binom{X_{k}}{1}\subseteq \mathcal{F}(X)$, it is sufficient to show $ |F|+|\phi(F)|\leq |\left\{ x\right\}|+|Y\backslash N(x)|$ for any $x\in X_{k}$, i.e.,
		\begin{equation} \tag{$2.1$}
			|F|-|N(F)|\leq 1-d(X_{k}).
		\end{equation}
		Since $d(X_{1})>d(X_{2})>\cdots>d(X_{k})$, we have
		\begin{equation*}
			d(X_{k})\leq d(X_{i_{1}})-k+i_{1}.
		\end{equation*}  
		Combing with $p\leq k-i_{1}+1$ by $\left\{ i_{1}, i_{2}, \cdots , i_{p}\right\}\subseteq \left\{ i_{1}, i_{1}+1, \cdots , k \right\}$, we get   $d(X_{k})-1\leq d(X_{i_{1}})-p$.	This  together with  $\left| N(F)\right|= |N(x_{i_{1}},  \cdots,  x_{i_{p}})|\geq |N(x_{i_{1}})|=d(X_{i_{1}})$ yields ($2.1$).
	\end{proof}	
	\begin{proof}[\textnormal{\textbf{Proof of Theorem \ref{import}}}]
		 Note that $|A|\leq |\phi(A)|$ for some $A\in \mathcal{F}(X,Y)$. Since Lemma \ref{c} and symmetry, there exists a $F\in \mathcal{F}(X,Y)$ such that $|F\cap O|\leq 1$ for any orbit $O$. Then, by  Lemma \ref{b}, equality $\alpha(X,Y)=\max \left\{1+|Y|-d(X_{k}),1+ |X|-d(Y_{\ell})\right\}$ holds.
		
		In the following, suppose that Theorem \ref{import} $(4)$ does not hold.  We divide our proof into the following two cases.
		
		\noindent \textbf{Case 1.} $|X|-d(Y_{\ell})\neq |Y|-d(X_{k})$. 
		
		We may assume $|X|-d(Y_{\ell})> |Y|-d(X_{k})$.
		In this case, it is sufficient to show that there is only singletons in $\mathcal{F}(X)$. Indeed, since $d(X_{1})>d(X_{2})>\cdots>d(X_{k})$, we have  $1+|Y|-d(X_{i})<1+|Y|-d(X_{k})$ for any $i\in[k-1]$, which implies $(\binom{X_{1}}{1}\cup \cdots \cup \binom{X_{k-1}}{1})\cap \mathcal{F}(X)=\emptyset$. 
		Therefore, if $\mathcal{F}(X)$ has only singletons, then $\emptyset\neq \mathcal{F}(X)\subseteq \binom{X_{k}}{1}$. Moreover, since $X_{k}$ is an orbit under the action of $\Gamma$, we have $\binom{X_{k}}{1}\subseteq \mathcal{F}(X)$.  Hence $(1)$ holds
		
		Let $A$ be a maximum-sized fragment in $\mathcal{F}(X)$. Then $B=Y\backslash N(A)$ is a minimum-sized fragment in $\mathcal{F}(Y)$. By $|X|-d(Y_{\ell})<|Y|-d(X_{k})$, we have $|B|>1$. Suppose $|A|>1$.   Since $(4)$ does not hold, there exist $p\in [k]$ and $z\in [\ell]$ such that $|A\cap X_{p}|>1$ and $|B\cap Y_{z}|>1$. The primitivity of $\Gamma\mid _{X_{p}}$ and $\Gamma\mid _{Y_{z}}$ implies that  $\emptyset\neq \gamma(A)\cap A\neq A$ and  $\emptyset\neq \eta(B)\cap B\neq B$ for some $\gamma,   \eta \in \Gamma$.  If $|A|\leq |B|$, then  $\gamma(A)\cup A\in \mathcal{F}(X)$ by Lemma \ref{4}, a contradiction to the fact that $|A|$ is the maximum-sized fragment in $\mathcal{F}(X)$. If $|A|>|B|$, then  $\eta(B)\cap B\in \mathcal{F}(Y)$   by Lemma \ref{4}.  This contradicts the fact that  $|B|$ is the minimum-sized fragment in $\mathcal{F}(Y)$. Hence $|A|=1$, which implies that  $\mathcal{F}(X)$ has only singletons.
		
		Similar arguments apply to $|X|-d(Y_{\ell})> |Y|-d(X_{k})$, we  deduce $(2)$ holds.
		
		\noindent \textbf{Case 2.} $|X|-d(Y_{\ell})=|Y|-d(X_{k})$. 
		
		Note that  $\alpha(X,  Y)=\max \left\{1+|Y|-d(X_{k}),  1+|X|-d(Y_{\ell})\right\}$. This together with  $|X|-d(Y_{\ell})=|Y|-d(X_{k})$ and Lemma \ref{3} \textnormal{(i)} yields $(3)$.
	\end{proof}
	 	Suppose that $G(X,Y)$,  $\Gamma$ are as Theorem \ref{import} and $|X|-d(Y_{\ell})=|Y|-d(X_{k})$.  Then   the fragment not in $\binom{X_{k}}{1}  \cup\binom{Y_{\ell}}{1} \cup \phi \left(\binom{X_{k}}{1}\cup \binom{Y_{\ell}}{1}\right)$ is said to be \textit{nontrivial}.	
	
	\begin{prop}\label{8}
		Let $G(X,Y)$ and $\Gamma$ be as in Theorem \ref{import}. If $|X|-d(Y_{\ell})=|Y|-d(X_{k})$ and there exists nontrivial fragment in $\mathcal{F}(X,Y)$, then the minimum-sized nontrivial fragment $F$ satisfies the following:
		\begin{enumerate}[\normalfont(i)]
			\vspace{-0.2cm}	\item $F$ is not a singleton;
			\vspace{-0.2cm}	\item $|F\cap X_{i}|\leq 1$ for any $i\in[k-1]$ and $|F\cap Y_{j}|\leq 1$ for any $j\in [\ell -1]$. 
		\end{enumerate}	
	\end{prop}    
	\begin{proof}[\textnormal{\textbf{Proof.}}]
	  Without loss of generality, we may assume $F\in\mathcal{F}(X)$.  
		
		\textnormal{(i)} Since $F$ is not a singleton of $X_{k}$ and $d(X_{1})>d(X_{2})>\cdots>d(X_{k})$,  we have  $|F|>1$.
		
		\textnormal{(ii)} It is clear $\phi(F)\in\mathcal{F}(X,Y)$ from Lemma \ref{3} \textnormal{(i)}. Note that $F\notin \binom{X_{k}}{1}\cup \binom{Y_{\ell}}{1} \cup \phi \left(\binom{X_{k}}{1}\cup \binom{Y_{\ell}}{1}\right)$ and $\phi^{-1}=\phi$. Then $\phi(F)$ is also a nontrivial fragment. The minimality of $F$ implies $|F|\leq |\phi(F)|$.
		
		Next, we prove that $|F\cap \gamma(F)|\in \lbrace 0,1, |F|\rbrace$ for any $\gamma\in \Gamma$. Otherwise there exists $\gamma \in \Gamma$ such that $1<|\gamma(F)\cap F|<|F|$. By Lemma \ref{4},  we have $\gamma(F)\cap F\in \mathcal{F}(X,Y)$. Moreover, we have $|\gamma(F)\cap F|<|F|\leq |\phi(x_{k})|=|\phi(y_{\ell})|$ for any $x_{k}\in X_{k}, y_{\ell}\in Y_{\ell}$. These imply that $\gamma(F)\cap F$ is a nontrivial fragment, which contradicts the fact that $F$ is the minimum-sized nontrivial fragment.

		It is clear that $|F\cap Y_{j}|=0$ for any $j\in[\ell]$  since $F\subseteq X$. Suppose that   $|F\cap X_{p}|\geq 2$ for some $p\in[k-1]$.   Then there exists $\gamma \in \Gamma$ such that $\emptyset \neq F\cap \gamma(F) \cap X_{p}\neq F \cap X_{p}$. Together with $|F\cap \gamma(F)|\in \lbrace 0,1, |F|\rbrace$, we conclude that  $F\cap \gamma(F)$ is a singleton in $X_{p}$.   According to Lemma \ref{4},  we have $F\cap \gamma(F)\in \mathcal{F}(X)$, a contradiction to the fact  $1+|Y|-d(X_{p})<1+|Y|-d(X_{k})$.
	\end{proof}

	\begin{prop}\label{5}
		Let $G(X,Y)$ and $\Gamma$ be as in Theorem \ref{import}. If $d(X_{k})<d(X_{k-1})-1$ and $d(Y_{\ell})<d(Y_{\ell-1})-1$, then Theorem \ref{import} \textnormal{(iv)} does not hold. 
	\end{prop}	
	\begin{proof}[\textnormal{\textbf{Proof.}}]
	Suppose that $F$ is a set satisfying $|F|\geq2$ and $|F\cap P|\leq 1$ for any non-singleton orbit $P$ under the action of $\Gamma$. 	Next, we prove such $F$ not in $\mathcal{F}(X,Y)$. 
	We may assume  $F\subseteq X$, the proof of $F\subseteq Y$ is similar. 
	
	Set $F=\left\{  x_{i_{1}},  \cdots,  x_{i_{p}} \right\}$, where  $i_{1}< i_{2}< \cdots < i_{p}$, $x_{i_{1}}\in X_{i_{1}},  \cdots, x_{i_{p}}\in X_{i_{p}} $. If $N(F)=Y$, then $F \notin \mathcal{F}(X,Y)$ by the definition of $\mathcal{F}(X,Y)$. 
		In the following, suppose $N(F)\subsetneq Y$.
		
		  In order to prove $F \notin \mathcal{F}(X,Y)$, we just need to check that $|F|+|Y\backslash N(F)|<|\left\{ x\right\}|+|Y\backslash N(x)|$ where $x\in X_{k}$, i.e., 
		\begin{equation} \tag{$2.2$}
			|F|-|N(F)|<1-d(X_{k}).
		\end{equation}
		
		Since $d(X_{1})>d(X_{2})>\cdots>d(X_{k})$ and $d(X_{k})<d(X_{k-1})-1$, we have
		\begin{equation}
			\tag{2.3}
			d(X_{k})< d(X_{i_{1}})-k+i_{1}.
		\end{equation}  
		
		Note that $\left\{ i_{1}, i_{2}, \cdots , i_{p}\right\}\subseteq \left\{ i_{1}, i_{1}+1, \cdots , k \right\}$, which implies  $p\leq k-i_{1}+1$. Combining with (2.3), we have   $d(X_{k})-1< d(X_{i_{1}})-p$.
		This together with  $\left| N(F)\right|= |N(x_{i_{1}},  \cdots,  x_{i_{p}})|\geq |N(x_{i_{1}})|=d(X_{i_{1}})$ yields ($2.2$).
	\end{proof}

	\section{Proof of Theorem \ref{1}}
		
	Let $n$ and $t$ be positive integers.  Suppose that $R=\left\{ r_{1},   r_{2},   \ldots,   r_{k} \right\}$ and   $S=\left\{  s_{1},  s_{2},  \ldots,  s_{\ell}\right\}$ are subsets of $[n]$ with $r_{1}< r_{2}< \cdots< r_{k}$,   $s_{1}< s_{2}< \cdots< s_{\ell}$, $t\leq \min R\cup S$ and $r_{k}+s_{\ell}-t<n$. 
	Set $X_{i} =\binom{[n]}{r_{i}}$, $Y_{j}=\binom{[n]}{s_{j}}$, $X=X_{1}\cup\cdots \cup X_{k}$ and $Y=Y_{1}\cup\cdots \cup Y_{\ell}.$  
	
	To prove Theorem \ref{1}, we consider a bipartite graph $G_{1}(X,Y)$, whose vertex set is $X\biguplus Y$, and two vertices $x\in X$ and $y\in Y$ are adjacent if and only if $|x\cap y|<t$.  
	
	It is easy to verify that $S_{n}$ is an automorphism group of $G_{1}(X, Y)$ stabilizing $X$. Moreover, $X_{1},\ldots, X_{k}$ and   $Y_{1},\ldots ,Y_{\ell}$ are exactly all orbits under the action of $S_{n}$ on $X\cup Y$. Denote  the induced subgraphs of $G_{1}(X,  Y)$ with vertex sets $X_{i}\cup Y_{j}$ and $X_{i}\cup Y$ by $G_{1}(X_{i},   Y_{j})$ and $G_{1}(X_{i},   Y)$, respectively. 
	
	\begin{prop}\label{6}
		Let $G_{1}(X,  Y)$ be as above. Then the following hold. 
		\begin{enumerate}[\normalfont(i)]
			\vspace{-0.3cm}	\item $G_{1}(X,  Y)$ is  a non-complete bipartite graph.
			\vspace{-0.3cm}	\item  $|Y_{j}\backslash N(x_{i})|\leq |Y_{j}\backslash N(x_{i+1})|-t(s_{j}-t+1)$ for any $i\in [k-1]$, $j\in[\ell]$, $x_{i}\in X_{i}$ and  $x_{i+1}\in X_{i+1}$.  Moreover, $|Y\backslash N(x_{i})|\leq |Y\backslash N(x_{i+1})|-\sum_{j=1}^{\ell}t(s_{j}-t+1)$.  In particular, $|Y\backslash N(x_{i})|\leq |Y\backslash N(x_{i+1})|-2 $ if $s_{\ell}>1$ and $|Y\backslash N(x_{i})|\leq |Y\backslash N(x_{i+1})|-(\ell+1)$  if $\ell>1$.
			\vspace{-0.3cm}	\item  If $r_{i}\geq s_{j}$,   then $|Y_{j}\backslash N(x_{i},  x_{i}^{\prime})|\leq |Y_{j}\backslash N(x_{i})|-1$ for any $x_{i},  
			x_{i}^{\prime}\in X_{i}$.
			\vspace{-0.3cm}	\item If $n=2r_{i}$, then $|Y_{j}\backslash N(x_{i},  \overline{x_{i}})|\leq |Y_{j}\backslash N(x_{i})|-1$ for any $j\in [\ell]$ and $x_{i}\in X_{i}$, where $\overline{x_{i}}=[n]\backslash x_{i}$. Moreover, $|Y_{j}\backslash N(x_{i},  \overline{x_{i}})|\leq |Y_{j}\backslash N(x_{i})|-2$ if $r_{i}\neq s_{j}$ or $r_{i}=s_{j}>t>1$.
		\end{enumerate}
	\end{prop}

	\begin{proof}[\textnormal{\textbf{Proof.}}]
		\textnormal{(i)} Observe that $[r_{1}]\in X_{1}$ and  $[s_{1}]\in Y_{1}$ are not adjacent.  Hence $G_{1}(X,  Y)$ is not a complete bipartite graph.
		
		\textnormal{(ii)} Since $X_{i}$ and $X_{i+1}$ are two orbits under the action of $S_{n}$, we may assume that $x_{i}=[r_{i}]$ and   $x_{i+1}=[r_{i+1}]$.  Then $Y_{j}\backslash N(x_{i})\subseteq Y_{j}\backslash N(x_{i+1}) $. 	
		Consider the following non-empty set
		$$\mathcal{D}=\left\{ z_{1}\cup z_{2} \cup z_{3} :z_{1}\in \binom{[r_{i}]}{t-1},   z_{2}\in \binom{[r_{i+1}]\backslash [r_{i}]}{1},   z_{3}\in \binom{[n]\backslash [r_{i+1}]}{s_{j}-t}\right\}. $$
		It is easy to check that 
		$\mathcal{D}\subseteq (Y_{j}\backslash N(x_{i+1}))\backslash (Y_{j}\backslash N(x_{i}))$ and $|\mathcal{D}|\geq t(s_{j}-t+1)$. Hence $|Y_{j}\backslash N(x_{i})|\leq |Y_{j}\backslash N(x_{i+1})|- t(s_{j}-t+1)$.	Moreover, we have
		\begin{equation}\tag{3.1}
			\begin{aligned}
				|Y\backslash N(x_{i})|=&\sum_{j=1}^{\ell}|Y_{j}\backslash N(x_{i})|
				\leq \sum_{j=1}^{\ell}|Y_{j}\backslash N(x_{i+1})|-t(s_{j}-t+1)\\
				=&|Y\backslash N(x_{i+1})|-\sum_{j=1}^{\ell}t(s_{j}-t+1).
			\end{aligned}
		\end{equation}
		
		If $s_{\ell}>1$,by $s_{\ell}\geq t$, then $t(s_{\ell}-t+1)\geq 2 $. Therefore, we have $|Y\backslash N(x_{i})|\leq |Y\backslash N(x_{i+1})|-2 $ by (3.1). If $\ell> 1$, then $s_{\ell}>1$, which implies $t(s_{\ell}-t+1)\geq 2 $. Combing with $t(s_{j}-t+1)\geq 1$ for any $j\in[\ell-1]$, we have $|Y\backslash N(x_{i})|\leq |Y\backslash N(x_{i+1})|-(\ell+1)$ by (3.1).

		\textnormal{(iii)}  Note that $Y_{j}\backslash N(x_{i},  x_{i}^{\prime}) \subseteq Y_{j}\backslash N(x_{i})$.  Then it is sufficient to prove  $Y_{j}\backslash N(x_{i},  x_{i}^{\prime}) \subsetneq Y_{j}\backslash N(x_{i})$.  
		
		If $|x_{i}\cap x_{i}^{\prime}|<t$, by $r_{i}\geq s_{j}$, then $\binom{x_{i}}{s_{j}} \subseteq ( Y_{j}\backslash N(x_{i}))$$\backslash (Y_{j}\backslash N(x_{i},  x_{i}^{\prime}))$. If $|x_{i}\cap x_{i}^{\prime}|\geq t$ and $r_{i}-|x_{i}\cap x_{i}^{\prime}|\geq s_{j}-t+1$,  consider the following non-empty set
		$$\mathcal{G}=\left\{ z_{1}\cup z_{2} :z_{1}\in \binom{x_{i}\cap x_{i}^{\prime}}{t-1},    z_{2}\in \binom{x_{i}\backslash x_{i}^{\prime}}{s_{j}-t+1}\right\}. $$
		It is easy to verify that $\mathcal{G}\subseteq(Y_{j}\backslash N(x_{i}))  \backslash (Y_{j}\backslash N(x_{i}, x_{i}^{\prime}))$. 
		If $|x_{i}\cap x_{i}^{\prime}|\geq t$ and $r_{i}-|x_{i}\cap x_{i}^{\prime}|<s_{j}-t+1$,  consider the following non-empty set
		$$\mathcal{H}=\left\{ z_{1}\cup z_{2} \cup x_{i}\backslash x_{i}^{\prime} :z_{1}\in \binom{x_{i}\cap x_{i}^{\prime}}{t-1},    z_{2}\in \binom{[n]\backslash (x_{i}\cup x_{i}^{\prime})}{s_{j}-|x_{i}\backslash x_{i}^{\prime}|-t+1}\right\}. $$
		We  also verify that $\mathcal{H}\subseteq(Y_{j}\backslash N(x_{i}))  \backslash (Y_{j}\backslash N(x_{i}, x_{i}^{\prime}))$. 
		
		\textnormal{(iv)} Note that $Y_{j}\backslash N(x_{i},  \overline{x_{i}})\subseteq Y_{j}\backslash N(x_{i})$. Then 
		$$|Y_{j}\backslash N(x_{i})|-|Y_{j}\backslash N(x_{i},  \overline{x_{i}})|=\left|  \left(Y_{j}\backslash N(x_{i}) \right)    \right\backslash \left(Y_{j}\backslash N(x_{i},  \overline{x_{i}}) \right)|.$$ 
		If $s_{j}=r_{i}$,  then $x_{i}\in (Y_{j}\backslash N(x_{i}))\backslash (Y_{j}\backslash N(x_{i},   \overline{x_{i}}))$. Thus $|Y_{j}\backslash N(x_{i},  \overline{x_{i}})|\leq |Y_{j}\backslash N(x_{i})|-1$. 
		If $s_{j}=r_{i}>t>1$,  then consider $$\mathcal{C}_{1}=\left\{ C\cup \left\{ z \right\}: C\in \binom{x_{i}}{r_{i}-1},   z\in \overline{x_{i}} \right\}\subseteq (Y_{j}\backslash N(x_{i}))\backslash (Y_{j}\backslash N(x_{i},   \overline{x_{i}})).$$      
		If $r_{i}>s_{j}$, then consider $$\mathcal{C}_{2}=\binom{x_{i}}{s_{j}}\subseteq (Y_{j}\backslash N(x_{i}))\backslash (Y_{j}\backslash N(x_{i},   \overline{x_{i}})).$$ 
		If $r_{i}<s_{j}$,  then consider
		$$\mathcal{C}_{3}=\left\{ x_{i}\cup \binom{\overline{x_{i}}}{s_{j}-r_{i}}\right\}\subseteq (Y_{j}\backslash N(x_{i}))\backslash (Y_{j}\backslash N(x_{i},   \overline{x_{i}})).$$ 
		It is clear that $|\mathcal{C}_{1}|, |\mathcal{C}_{2}|, |\mathcal{C}_{3}|\geq 2$. The desired result follows.
	\end{proof}   

		Note that 
		$|Y\backslash N(x_{i})|=|Y|-d(X_{i})$ for any  $i\in [k]$ and $x_{i}\in X_{i}$. By $|Y\backslash N(x_{i})|\leq |Y\backslash N(x_{i+1})|-\sum_{j=1}^{\ell}t(s_{j}-t+1) $, we have $d(X_{i+1})\leq d(X_{i})-\sum_{j=1}^{\ell}t(s_{j}-t+1)$ for any $i\in[k-1]$. Then  $d(X_{1})> d(X_{2})>\cdots >d(X_{k})$. Furthermore, we have $d(X_{k})< d(X_{k-1})-1 $ if $s_{\ell}>1$.

		By computation, we have
		\begin{equation*}
			\begin{aligned}
				|Y|-d(X_{k})=&\sum_{j=t}^{s_{1}}\binom{r_{k}}{j}\binom{n-r_{k}}{s_{1}-j}+\cdots+\sum_{j=t}^{s_{\ell}}\binom{r_{k}}{j}\binom{n-r_{k}}{s_{\ell}-j},
			\end{aligned}
		\end{equation*}
		\begin{equation*}
			\begin{aligned}
				|X|-d(Y_{\ell})=&\sum_{j=t}^{r_{1}}\binom{s_{\ell}}{j}\binom{n-s_{\ell}}{r_{1}-j}+\cdots+\sum_{j=t}^{r_{k}}\binom{s_{\ell}}{j}\binom{n-s_{\ell}}{r_{k}-j}.
			\end{aligned}
		\end{equation*}
So, in order to prove inequality $(1.1)$, we just need to prove $\alpha(X,Y)=\max \lbrace 1+|Y|-d(X_{k}), 1+|X|-d(Y_{\ell})
 \rbrace$.
 	\begin{proof}[\textnormal{\textbf{Proof of inequality $(1.1)$.}}]
		If $n=2$, then $r_{k}=s_{\ell}=t=1$ by $n>r_{k}+s_{\ell}-t$. It is clear that $(1.1)$ holds.	Next, we may assume $n\geq 3$.
		
		By Proposition \ref{6} \textnormal{(ii)} and symmetry,  both $d(X_{1})>d(X_{2})>\cdots>d(X_{k})$ and $d(Y_{1})>d(Y_{2})>\cdots>d(Y_{\ell})$ hold. For each $i\in [k]$, in order to prove that $X_{i}$ is not contained in any nontrivial independent set, we just need to check that $N(X_{i})=Y$.  For any $z\in[\ell]$, since $r_{i}+s_{z}-t<n$, we know that $[s_{z}]\in Y_{z}$ is adjacent to $$x_{i}=\left\{ 1,2,\cdots,t-1, s_{z}+1,\cdots, s_{z}+r_{i}-(t-1) \right\}\in X_{i}.$$   
		Hence $[s_{z}]\in N(X_{i})$.   
	    For any $y\in Y_{z}$, there exists $\gamma \in S_{n}$ such that $y=\gamma([s_{z}])$ since $Y_{z}$ is an orbit under the action of $S_{n}$. Then $y$ is adjacent to $\gamma(x_{i})\in X_{i}$. Consequently $Y_{z}\subseteq N(X_{i})$.  Since $z$ is arbitrary, we have $N(X_{i})=Y$. Similarly, for each $j\in [\ell]$, $Y_{j}$ is not contained in any nontrivial independent set.  	
		
		If $n\neq 2d$, it is well known that for each $D\in\binom{n}{d}$,  the stabilizer of $D$ is a maximal subgroup of $S_{n}$ \cite{MR2265507}.
		Thus $S_{n}\mid_{X_{i}}$,  $S_{n}\mid_{Y_{j}}$ are primitive if  $n\neq 2r_{i}$  for any $i\in[k]$ and $ n\neq 2s_{j}$  for any  $   j\in[\ell]$.  By Theorem \ref{import},  equality $\alpha(X,Y)=\max \lbrace 1+|Y|-d(X_{k}), 1+|X|-d(Y_{\ell})
		\rbrace$ holds. 
		
		Now  assume $n=2r_{p}$ for some $p\in [k]$ or $n=2s_{z}$ for some $z\in[\ell]$.    Note that there exists  $A\in \mathcal{F}(X,Y)$ such that $|A|\leq |\phi(A)|$. If $A\in\mathcal{F}(X)$, $n\neq 2r_{i}$ for any $i\in[k]$ or $A\in\mathcal{F}(Y)$, $n\neq 2s_{j}$ for any $j\in[\ell]$, then the desired result holds by Lemmas \ref{c} and  \ref{b}.  
		
		Next suppose $A\in \mathcal{F}(X)$, $n=2r_{p}$ for some $p\in [k]$ or $A\in \mathcal{F}(Y)$, $n=2s_{z}$ for some $z\in [\ell]$. Without loss of generality, we consider the case $A\in \mathcal{F}(X)$ and $n=2r_{p}$ for some $p\in [k]$. Then the only imprimitive sets of $S_{n}\mid_{X_{p}}$ are all pairs of complementary subsets. If $|A\cap X_{j}|\geq 2$ for some $j\neq p$. We conclude  $|A\cap X_{j}|<|X_{j}|$ from $X_{j}\nsubseteq A$. Since $S_{n}\mid _{X_{j}}$ is primitive, there exists $\gamma\in S_{n}$ such that $\emptyset \neq A\cap \gamma(A)\neq A$.  By Lemma \ref{4},  we get $A\cap \gamma(A)\in \mathcal{F}(X)$. 	Note that $|A\cap \gamma(A)|\leq |A|\leq |\phi(A)|\leq |\phi(A\cap \gamma(A))|$. Then we can use $A\cap \gamma(A)$ to replace $A$.  If $A\cap X_{p}\nsubseteq \left\{ x_{p},  \overline{x_{p}} \right\}$ for any $x_{p}\in X_{p}$. Since $A\cap X_{p}$ is not an imprimitive set in $X_{p}$, there exists $\eta\in S_{n}$ such that $\emptyset \neq A\cap \eta(A)\neq A$. Similarly, we can use $A\cap \eta(A)$ to replace $A$. By above discussion, we may assume that $|A\cap X_{i}|\leq 1$ for any   $i\neq p$ and $A\cap X_{p}\subseteq \left\{ x_{p},  \overline{x_{p}} \right\}$  for some $x_{p}\in X_{p}$.  Next, we divide our proof into the following cases.
		
		\noindent\textbf{Case 1.} $(k,\ell)=(1,1).$
		
		If $|A|=|A\cap X_{1}|=1$, then the there is nothing to prove. If $A=A\cap X_{1}=\lbrace x, \overline{x} \rbrace$, then $\alpha(X,Y)=1+|Y|-d(X_{1})$ also holds by the following Claim.
		\begin{claim}\label{a}
			
			If $k=\ell=1$,   $n=2r_{1}\geq 4$,   then $\left\{ x,  \overline{x} \right\}\in \mathcal{F}(X)$ if and only if $s_{1}=r_{1}$ and $t=1$. Moreover, if $\lbrace x, \overline{x} \rbrace\in \mathcal{F}(X)$, then $\left\{ x\right\}\in \mathcal{F}(X)$.  
		\end{claim}
		
		\begin{proof}[\textnormal{\textbf{Proof.}}]
			If $s_{1}=r_{1}=t$,  then $\phi(x,  \overline{x})=\emptyset$. Hence $\left\{ x,  \overline{x} \right\}\notin \mathcal{F}(X)$. 
			If $s_{1}=r_{1}>t$ and $t=1$, it is easy to check that $\left\{ x,  \overline{x}\right\}$ and $\left\{ x\right\}$ are also in $\mathcal{F}(X)$ by the structure of $G_{1}(X,Y)$. 
			
			On the other hand, we have $|Y\backslash N( x,  \overline{x})|\leq |Y\backslash N( x)|-2$ by Proposition \ref{6} \textnormal{(iv)}. Then 
			$$|\left\{ x,  \overline{x} \right\} |+|Y\backslash N( x,  \overline{x})| < |\left\{ x\right\}|+|Y\backslash N(x)|\leq \alpha(X,Y),$$
			which implies $\left\{ x,  \overline{x} \right\}\notin \mathcal{F}(X)$. The desired result follows.
		\end{proof}

		\noindent\textbf{Case 2.} $(k,\ell)\neq (1,1)$. 
		
		By Lemma \ref{b}, to prove (1.1), it is sufficient to show $|A\cap X_{p}|\leq 1$. Suppose for contradiction that $A\cap X_{p}=\lbrace x_{p},\overline{x_{p}}\rbrace$ for some $x_{p}\in X_{p}$.
		
		\noindent\textbf{Case 2.1.}
		$\ell>1$, $|A|=2$. 
		
		In this case, set $A=\lbrace x_{p}, \overline{x_{p}}\rbrace$.
		By Proposition \ref{6} \textnormal{(iv)},  we have $$|Y_{\ell}\backslash N(x_{p})|-|Y_{\ell}\backslash N(x_{p},   \overline{x_{p}})|\geq 1,\quad  |Y_{\ell-1}\backslash N(x_{p})|-|Y_{\ell-1} \backslash N(x_{p},   \overline{x_{p}})|\geq 1.$$ 
		Thus $|Y\backslash N(x_{p})|-|Y\backslash N(x_{p},   \overline{x_{p}})|\geq 2$, implying $|\left\{ x_{p},  \overline{x_{p}} \right\} |+|\phi( x_{p},  \overline{x_{p}})| < |\left\{ x_{p}\right\}|+|Y\backslash N(x_{p})|$, a contradiction to the fact   $A=\left\{ x_{p},  \overline{x_{p}} \right\}\in \mathcal{F}(X)$.

		\noindent\textbf{Case 2.2.} $\ell>1$, $|A|\geq 3$.  
		
		Let $i_{1}=\min \left\{i: A\cap X_{i}\neq \emptyset \right\}$ and $x_{i_1}\in A\cap X_{i_{1}}$. For any $m\in[k-i_{1}]$, choose some $z_{i_{1}+m}\in X_{i_{1}+m}$. By Proposition \ref{6} \textnormal{(ii)},  we have 
		\begin{equation*}
			\begin{aligned}
				|Y\backslash N(x_{i_{1}} )\leq
				|Y\backslash N(z_{i_{1}+1} )|-(\ell+1).
			\end{aligned}
		\end{equation*}
		Repeating this process successively for $(i_{1}+1,i_{1}+2), \ldots,(k-1, k)$, we get
		\begin{equation*}\tag{$3.2$}
			\begin{aligned}
				|Y\backslash N(x_{i_{1}} )\leq
				|Y\backslash N(z_{k} )|-(k-i_{1})(\ell+1).
			\end{aligned}
		\end{equation*}
		
		Since $\left\{i: A\cap X_{i}\neq \emptyset \right\}\subseteq \left\{ k, k-1, \ldots, i_{1}+1 ,i_{1}\right\}$ and $\left| \left\{i: A\cap X_{i}\neq \emptyset \right\}\right|=|A|-1$, we have $|A|-1\leq k-i_{1}+1$. This together with $(3.2)$ yields  $$|A|+|\phi(A)|\leq |A|+|Y\backslash N(x_{i_{1}})| \leq  |A|+|Y\backslash N(z_{k})|-(|A|-2)(\ell+1)< 1+|Y\backslash N(z_{k})|,$$
		which contradicts to $A\in \mathcal{F}(X)$.  
		
		\noindent	\textbf{Case 2.3.} $k>1, \ell=1$  and $  s_{\ell}=t$. 
		
		In this case, we have $N(A)=Y$,   a contradiction to the fact $A \in \mathcal{F}(X)$.

		\noindent	\textbf{Case 2.4.} $k>1,   \ell=1$ and $ s_{\ell}>t$. 
		
		Let $i_{1}=\min \left\{i: A\cap X_{i}\neq \emptyset \right\}$ and $x_{i_1}\in A\cap X_{i_{1}}$. For any $m\in[k-i_{1}]$, choose some $z_{i_{1}+m}\in X_{i_{1}+m}$. By Proposition \ref{6} \textnormal{(ii)},  we have 
		\begin{equation*}
			\begin{aligned}
				|Y\backslash N(x_{i_{1}} )\leq
				|Y\backslash N(z_{i_{1}+1} )|-2.
			\end{aligned}
		\end{equation*}
		Repeating this process successively for $(i_{1}+1,i_{1}+2), \ldots,(k-1, k)$, we get
		\begin{equation*}\tag{$3.3$}
			\begin{aligned}
				|Y\backslash N(x_{i_{1}} )\leq
				|Y\backslash N(z_{k} )|-2(k-i_{1}).
			\end{aligned}
		\end{equation*}
		
		Since $\left\{i: A\cap X_{i}\neq \emptyset \right\}\subseteq \left\{ k, k-1, \ldots, i_{1}+1 ,i_{1}\right\}$ and $\left| \left\{i: A\cap X_{i}\neq \emptyset \right\}\right|=|A|-1$, we have $|A|-1\leq k-i_{1}+1$.  This together with $(3.3)$ yields  $$|A|+|\phi(A)|\leq |A|+|Y\backslash N(z_{k})|-2(|A|-2),$$ which implies that  one of $|A|=2$ and $|A|=3$ holds. Otherwise $|A|+|\phi(A)|< 1+|Y\backslash  N(z_{k})|$.  
		
		\noindent\textbf{Case 2.4.1.} $|A|=2$.
		
		In this case $A=\left\{ x_{p} ,\overline{x_{p}} \right\}$. Then $|A\cap X_{p}|+|\phi(A)|=\alpha(X_{p},  Y)$.    Claim \ref{a} clearly forces  $s_{1}=r_{p}$ and $t=1$. Combining with $k>1$, we have $$|A|+|\phi(A)|=|\lbrace x_{p} \rbrace|+|Y\backslash N(x_{p})|=|\lbrace y\rbrace|+|X_{p}\backslash N(y)|<1+|X\backslash N(y)|,$$  where $y\in Y$.  This contradicts  the fact that $A\in \mathcal{F}(X)$.

		\noindent\textbf{Case 2.4.2.} $|A|=3.$
		
		Set $A\cap X_{q}=\left\{ x_{q}\right\}$ and $i=\min\lbrace p,q \rbrace$. It is clear that $i\leq k-1$.  By  $(3.3)$, we have $$|A|+|\phi(A)|\leq |A| +|Y\backslash N(z_{k})|-2(k-i)=3+|Y\backslash N(z_{k})|-2(k-i).$$
		Hence $i=k-1$, implying that $\left\{ p,   q \right\}=\left\{  k-1,  k   \right\}$.	If $p=k-1$, by Proposition \ref{6} \textnormal{(ii)} and \textnormal{(iv)}, we have
		\begin{equation*}
			\begin{aligned}
				|A|+|\phi(A)|&\leq |A|+|Y\backslash N(x_{p},   \overline{x_{p}})|
				\leq |A|+|Y\backslash N(x_{p})|-1\\&\leq |A|+|Y\backslash N(x_{q})|-3
				<1+ |Y\backslash N(x_{q})|, 
			\end{aligned}
		\end{equation*}
		a contradiction to the fact $A\in \mathcal{F}(X)$.  Thus $p=k$.
		
		If  $t\geq 2$, by Proposition \ref{6} \textnormal{(ii)},  we have  $$|\phi(A)|\leq |Y\backslash N(x_{q})|\leq |Y\backslash N(x_{p})|-t(s_{l}-t+1)\leq |Y\backslash N(x_{p})|-4.$$ If  $t= 1$ and  $s_{\ell}\geq 3$, by Proposition \ref{6} \textnormal{(ii)},  we have $$|\phi(A)|\leq |Y\backslash N(x_{q})|\leq |Y\backslash N(x_{p})|-t(s_{l}-t+1)\leq |Y\backslash N(x_{p})|-3.$$ If $t= 1$ and $s_{\ell}=2$,  then $|Y_{\ell}\backslash N(x_{p}, \overline{x_{p}})|=r_{k}^{2}$. Therefore, we have
		\begin{equation*}
			\begin{aligned}
				|A|+|\phi(A)|\leq 3+r_{k}^{2}
				<1+2\binom{2r_{k}-2}{r_{k-1}-1}+\binom{2r_{k}-2}{r_{k}-2}+2\binom{2r_{k}-2}{r_{k}-1}
				\leq 1+ |X\backslash N(y)|,  
			\end{aligned}
		\end{equation*}
		where $y\in Y$. In the last three cases, we have $|A|+|\phi(A)|<\alpha(X,Y)$, a contradiction.
	\end{proof}	
	 To prove Theorem \ref{1}  $(1)$, we need the following theorem.
	 \begin{thm}\textnormal{(\cite{MR2971702})}\label{ref1}
	 	Let $n$, $r$, $s$ and $t$ be positive integers with $n\geq 4$, $r,s \geq 2$, $t<\min\lbrace r,s \rbrace$, $r+s-t<n$, $(n,t)\neq (r+s,1)$  and $\binom{n}{r}\leq \binom{n}{s}$.  If ${A}\subseteq \binom{\left[n\right]}{r}$ and $\mathcal{B}\subseteq \binom{\left[n\right]}{s}$ are non-empty cross $t$-intersecting, then 
	 	$$\left|\mathcal{A}\right|+\left|\mathcal{B}\right|\le 	1+\binom{n}{s}-\sum_{i=0}^{t-1}\binom{r}{i}\binom{n-r}{s-i}.$$ 
	 	Moreover equality holds if and only if one of the following holds:
	 	\begin{enumerate}[\normalfont(i)]
	 		\vspace{-0.2cm}	\item $\mathcal{A}=\left\{ A \right\}$ and $\mathcal{B}=\left\{ B\in \binom{[n]}{s}: |B\cap A|\geq t\right\}$ for some $A\in\binom{[n]}{r}$;
	 		\vspace{-0.2cm}	\item $\binom{n}{r}=\binom{n}{s}$ and  $\mathcal{B}=\left\{ B \right\}$ and $\mathcal{A}=\left\{ A\in \binom{[n]}{r}: |A\cap B|\geq t\right\}$ for some $B\in \binom{[n]}{s}$;
	 		\vspace{-0.2cm}	\item $(r,s,t)=(2,2,1)$ and $\mathcal{A}=\mathcal{B}=\left\{  
	 		C\in \binom{[n]}{ 2}: i\in C\right\}$ for some $i\in [n]$;
	 		\vspace{-0.2cm}	\item $(r,s,t)=(n-2,n-2,n-3)$ and $\mathcal{A}=\mathcal{B}=\binom{A}{n-2}$ for some $A\in \binom{[n]}{n-1}$.
	 	\end{enumerate}
	 \end{thm}
	\begin{proof}[\textnormal{\textbf{Proof of Theorem \ref{1}  $(1)$.}}]
		There is nothing to prove when $r_{k}=t$. Thus we may assume that $r_{k}>t$.

		 The proof of the inequality $(1.1)$ shows that $G_{1}(X,  Y)$ satisfies the conditions of Theorem \ref{import} by $n>2r_{k}$.  
   	Then we have $\mathcal{F}(X,Y)\supseteq\binom{X_{k}}{1}  \cup\binom{Y_{\ell}}{1} \cup \phi \left(\binom{X_{k}}{1}\cup \binom{Y_{\ell}}{1}\right)$ since $|X|-d(Y_{\ell})=|Y|-d(X_{k})$. 
	Moreover, according to Proposition \ref{7}\textnormal{(ii)} and symmetry,  we get $d(X_{k-1})-1>d(X_{k})$ and $d(Y_{\ell-1})-1>d(Y_{\ell})$. By Proposition \ref{5},  Theorem \ref{import} \textnormal{(iv)} does not hold. 
		
		Now,  we  determine all the nontrivial fragments. Suppose that there is a nontrivial fragment in $\mathcal{F}(X,  Y)$.  By Lemma \ref{8},  there exists  non-singleton $F\in \mathcal{F}(X,  Y)$ such that $|F\cap X_{i}|\leq 1$ for any $i\in [k-1]$ and $|F\cap Y_{j}|\leq 1$ for any $j\in [\ell-1]$.  Moreover, we know that $F$ is a minimum-sized nontrivial fragment. By symmetry,   we may assume that $F\in \mathcal{F}(X)$. 	
		Since $|F|\geq 2$ and Theorem \ref{import} \textnormal{(iv)} does not hold, we have $|F\cap X_{k}|\geq 2$.  
		
		The definition of $\phi(F)$ implies $$\bigcup_{C\in \phi(F)}\left\{ B\in \binom{[n]}{r_{k}}:  C\subseteq B\right\}\subseteq \phi(F).$$ Thus
		\begin{equation}\tag{3.4}
			\phi(F)\cap Y_{k}\neq \emptyset.
		\end{equation}
		
		Note that both $G_{1}(X,  Y)$ and $G_{1}(X_{k},   Y_{k})$ satisfy all conditions in Theorem \ref{import}. By Theorem \ref{import} \textnormal{(iii)},  the elements of $\binom{X_{k}}{1}$  are fragments of $G_{1}(X,Y)$ and $G_{1}(X_{k},Y_{k})$. Consequently, for any $z\in X_{k}$, 
		\begin{equation}\tag{3.5}
			\alpha(X,  Y)=\alpha(X_{k},  Y_{k})+|Y_{k-1}\backslash N(z)|+\cdots+|Y_{1}\backslash N(z)|.
		\end{equation}
		
		Let $x_{k},   x_{k}^{\prime}\in F \cap X_{k}$. By $(3.4)$ and $(3.5)$, we have	
		\begin{equation}
			\tag{$3.6$}
			\begin{aligned}
				&\ |F\cap X_{k}|+|\phi(F)\cap Y_{k}|+|Y_{k-1}\backslash N(x_{k})|+\cdots+|Y_{1}\backslash N(x_{k})|\\
				\leq &\ \alpha(X_{k},  Y_{k})+|Y_{k-1}\backslash N(x_{k})|+\cdots+|Y_{1}\backslash N(x_{k})|
				=\alpha(X,   Y)=|F|+|\phi(F)|\\
				=&\ |F\cap X_{k}|+|\phi(F)\cap Y_{k}|+|F
				\backslash X_{k}| +|Y_{k-1}\backslash N(F)|+\cdots+|Y_{1}\backslash N(F)|.
			\end{aligned}
		\end{equation}
		In order to determine the structure of $F$, we consider inequality (3.6) in two cases.
		
		\noindent\textbf{Case 1.} $F\backslash X_{k}\neq \emptyset$. 
		
		In this case, we have $k\geq2$. 
		Let $i_{1}=\min\lbrace i\in [k-1]:F\cap X_{i}\neq \emptyset\rbrace$  and $x_{i_{1}}\in F\cap X_{i_{1}}$. For any $m\in[k-i_{1}]$, choose some $z_{i_{1}+m}\in X_{i_{1}+m}$. Then for any $j\in[\ell]$, by Proposition \ref{6} \textnormal{(ii)},  we have 
		\begin{equation*}
			\begin{aligned}
				|Y_{j}\backslash N(x_{i_{1}} )\leq
				|Y_{j}\backslash N(z_{i_{1}+1} )|-1.
			\end{aligned}
		\end{equation*}
		Repeating this process successively for $(i_{1}+1,i_{1}+2), \ldots,(k-1, k)$, we get
		\begin{equation*}\tag{$3.7$}
			\begin{aligned}
				|Y_{j}\backslash N(x_{i_{1}} )\leq
				|Y_{j}\backslash N(z_{k} )|-(k-i_{1}).
			\end{aligned}
		\end{equation*}
		
		Since $\lbrace i\in [k-1]:F\cap X_{i}\neq \emptyset\rbrace\subseteq \left\{  k-1,\ldots,i_{1}+1, i_{1}\right\}$ and $\left| \lbrace i\in [k-1]:F\cap X_{i}\neq \emptyset\rbrace\right|=|F\backslash X_{k}|$, we have $|F\backslash X_{k}|\leq k-i_{1}$.  This together with $(3.7)$ yields  $$|Y_{j}\backslash N(F)|\leq|Y_{j}\backslash N(x_{i_{1}} )|\leq
		|Y_{j}\backslash N(z_{k} )|-|F\backslash X_{k}|=|Y_{j}\backslash N(x_{k} )|-|F\backslash X_{k}|,$$ 
		which implies that inequality (3.6) transforms into 
		\begin{equation*}
			\begin{aligned}
				&\ |F\cap X_{k}|+|\phi(F)\cap Y_{k}|+|Y_{k-1}\backslash N(x_{k})|+\cdots+|Y_{1}\backslash N(x_{k})|\\
				\leq &\ \alpha(X_{k},  Y_{k})+|Y_{k-1}\backslash N(x_{k})|+\cdots+|Y_{1}\backslash N(x_{k})|
				\\		
				=&\ |F\cap X_{k}|+|\phi(F)\cap Y_{k}|+|Y_{k-1}\backslash N(x_{k})|+\cdots+|Y_{1}\backslash N(x_{k})|-(k-2)|F\backslash X_{k}|.
			\end{aligned}
		\end{equation*}
		Thus  $k=2$ and  above 
		 equality holds. Therefore, we have $|F\cap X_{2}|+|\phi(F)\cap Y_{2}|=\alpha(X_{2},  Y_{2})$. 
		
		Next, we  prove that $|\phi(F)\cap Y_{2}|\geq 2$. If $\phi(F)\cap Y_{2}=\lbrace y_{2}\rbrace$ for some  $y_{2}\in Y_{2}$,  then $\phi(F)\cap Y_{1}=\emptyset$. Otherwise  
		$$|F\cap X_{2}| \leq|X_{2}\backslash N(\phi(F)\cap Y_{1})|<|X_{2}\backslash N(y_{2})|=|F\cap X_{2}|.$$   Hence $\phi(F)=\left\{ y_{2}\right\}.$ This is a contradiction to the fact that  $F$  is a nontrivial fragment in $\mathcal{F}(X,  Y)$. 
		
		Note that $n> 2r_{2}$ and $r_{2}>t$. Then $n$, $r_{2}$ and $t$ satisfy conditions in Theorem \ref{ref1}. We have known that $F\cap X_{2}$ and $\phi(F)\cap Y_{2}$ are non-empty cross $t$-intersecting families with $|F\cap X_{2}|\geq 2$ and $|\phi(F)\cap Y_{2}|\geq 2$, and $|F\cap X_{2}|+|\phi(F)\cap Y_{2}|$ is maximum. Therefore, $F\cap X_{2}$ and $\phi(F)\cap Y_{2}$ are one of families stated in Theorem \ref{ref1}, but \textnormal{(i)} and \textnormal{(ii)}.  It is easy to check that Theorem \ref{ref1} \textnormal{(iv)} does not hold by $n>2r_{2}$.   Consequently, we have 
		$$F\cap X_{2}=\phi(F)\cap Y_{2}=\left\{  
		C\in \binom{[n]}{ 2}: i\in C\right\},$$for some $i\in [n]$.    Moreover, it is easy to check that  $F\cap X_{1}=\phi(F)\cap Y_{1}=\left\{ i \right\}$.

		\noindent	\textbf{Case 2.} $F\backslash X_{k}=\emptyset$.   
		
		By Proposition \ref{6} \textnormal{(iii)} and $|Y_{j}\backslash N(F)|\leq |Y_{j}\backslash N(x_{k},x_{k}^{\prime})|$,  inequality (3.6) transforms into 
		\begin{equation*}
			\begin{aligned}
				&\ |F\cap X_{k}|+|\phi(F)\cap Y_{k}|+|Y_{k-1}\backslash N(x_{k})|+\cdots+|Y_{1}\backslash N(x_{k})|\\
				\leq &\ \alpha(X_{k},  Y_{k})+|Y_{k-1}\backslash N(x_{k})|+\cdots+|Y_{1}\backslash N(x_{k})|\\
				=&\ |F\cap X_{k}|+|\phi(F)\cap Y_{k}|+|Y_{k-1}\backslash N(x_{k})|+\cdots+|Y_{1}\backslash N(x_{k})|-(k-1).
			\end{aligned}
		\end{equation*}
		Hence $k=1$ and $F$ is a nontrivial fragment in $G_{1}(X_{1}, Y_{1})$. 
		
		Note that $n> 2r_{1}$ and $r_{1}>t$. Then $n$, $r_{1}$ and $t$ satisfy conditions in Theorem \ref{ref1} and Theorem \ref{ref1} \textnormal{(iv)} does not hold. Similarly, we know that $F$ and $\phi(F)$ are non-empty cross $t$-intersecting families with $|F|\geq 2$ and $|\phi(F)|\geq 2$, and $|F|+|\phi(F)|$ is maximum.  Consequently, we have,  for some $i\in [n]$, 
		$$F=\phi(F)=\left\{	C\in \binom{[n]}{ 2}: i\in C\right\}.$$ 
		
		In above two cases, we have $|F|=|\phi(F)|$.  This together with the minimality of $F$ yields $|D|=|\phi(D)|$ for any nontrivial fragment $D$. That is to say,  the size of any nontrivial fragment is the minimum.   Applying above process for every nontrivial fragment, we get  the desired result.
	\end{proof}
	\begin{proof}[\textnormal{\textbf{Proof of Theorem \ref{1} $(2)$.}}]
		The Proof of inequality $(1.1)$ shows that $G_{1}(X,  Y)$ satisfies the conditions of Theorem \ref{import} by $n>2\max \lbrace r_{k}, s_{\ell}\rbrace $. 
		
		Next, we  check that  Theorem \ref{import} \textnormal{(iv)} does not hold. 	 If $(k, \ell)=(1,1)$, then there is nothing to prove.   
		On the other hand, we may assume $(k, \ell)\neq(1,1)$. If $r_{k}$,   $s_{\ell}\geq 2$,  by Proposition \ref{6} \textnormal{(ii)} and symmetry,  we have $d(X_{k-1})-1>d(X_{k})$ and  $d(Y_{\ell-1})-1>d(Y_{\ell})$. Hence   Theorem \ref{import} \textnormal{(iv)} does not hold by Proposition \ref{5}. If $r_{k}=1$ and $s_{\ell}\geq 2$,   suppose that $F$ is the fragment described in Theorem \ref{import} \textnormal{(iv)}. Then $F\in \mathcal{F}(Y)$ and $\ell\geq 2$. Since $2s_{\ell}+1\leq n$ and $s_{\ell}>1$, we have
		$$|F|+|\phi(F)|\leq \ell + s_{\ell}<n\leq 1+ \binom{n-1}{s_{\ell}-1}\leq |\lbrace x \rbrace|+|Y_{\ell}\backslash N(x)|\leq \alpha(X,   Y),$$ a contradiction to the fact $F\in \mathcal{F}(Y)$.  Therefore,  Theorem \ref{import} \textnormal{(iv)} does not hold. Similarly, we can prove that Theorem \ref{import} \textnormal{(iv)} does not hold if $r_{k}\geq2$ and $s_{\ell}=1$.
		
		By Theorem \ref{import}, we just need to prove that $|Y|-d(X_{k})> |X|-d(Y_{\ell})$ when $\max R\bigtriangleup S \in S$ and $|Y|-d(X_{k})< |X|-d(Y_{\ell})$ when $\max R\bigtriangleup S \in R$. We only prove the former,  the proof of the latter is similar. Let $s_{\ell-p}=\max R\bigtriangleup S$,  then $r_{k-p+1}=s_{\ell-p+1}$, $r_{k-p+2}=s_{\ell-p+2}$,     $\ldots$, $r_{k}=s_{\ell}$.  If $k=p$, then it is easy to see $|Y|-d(X_{k})> |X|-d(Y_{\ell})$. Next, we may assume $k>p$.

		 It is clear that $s_{\ell -p}> r_{k-p}\geq t$. This together with $n\geq  2^{s_{\ell}+1}k(s_{\ell}-t)+2r_{k}+1$ yields 
		\begin{equation}\tag{3.8}
			\begin{aligned}
				&\binom{r_{k}}{t}\binom{n-r_{k}}{s_{\ell-p}-t}=\binom{r_{k}}{t}\binom{n-r_{k}-1}{s_{\ell-p}-t-1}\dfrac{n-r_{k}}{s_{\ell-p}-t}\geq \binom{n-r_{k}-1}{s_{\ell-p}-t-1}\dfrac{n-r_{k}}{s_{\ell-p}-t}\\
				\geq&\ 2k\binom{n-r_{k}-1}{s_{\ell-p}-t-1}\left( \binom{s_{\ell}}{0}+\binom{s_{\ell}}{1}+\cdots + \binom{s_{\ell}}{s_{\ell}}  \right)\\
				\geq&\ 2k\binom{n-r_{k}-1}{s_{\ell-p}-t-1}\left( \binom{s_{\ell}}{t}+\binom{s_{\ell}}{t+1}+\cdots + \binom{s_{\ell}}{r_{k}}  \right)+1.
			\end{aligned}
		\end{equation}
		
		Note that $n-r_{k}-1\geq 2(s_{\ell-p}-t-1)$ and $s_{\ell-p}\geq r_{k-p}+1$. Then
		\begin{equation}\tag{3.9}
			\binom{n-r_{k}-1}{s_{\ell-p}-t-1}\geq \binom{n-r_{k}-1}{r_{k-p}-t}
		\end{equation}
		
		It is easy to check $s_{\ell}\geq r_{k}$ and $n-r_{k}-r_{k-p}+t\geq r_{k-p}-t$, implying that 
		\begin{equation}\tag{3.10}
			\begin{aligned}
				&\binom{n-s_{\ell}}{r_{k-p}-t}\leq \binom{n-r_{k}}{r_{k-p}-t}=\dfrac{n-r_{k}}{n-r_{k}-r_{k-p}+t}\binom{n-r_{k}-1}{r_{k-p}-t}\\=& \left(1+\dfrac{r_{k-p}-t}{n-r_{k}-r_{k-p}+t}\right)\binom{n-r_{k}-1}{r_{k-p}-t} \leq 2\binom{n-r_{k}-1}{r_{k-p}-t}.
			\end{aligned}
		\end{equation}
		
		Note that $r_{k}\geq r_{i}$, $n-s_{\ell}\geq 2(r_{i}-t)$ and $r_{i}-t\geq r_{i}-j$ for any $i\in [k-p]$ and $t\leq j\leq r_{i}$. Then we have
		\begin{equation*}
			\begin{aligned}
				&\binom{n-s_{\ell}}{r_{k-p}-t}\left( \binom{s_{\ell}}{t}+\binom{s_{\ell}}{t+1}+\cdots + \binom{s_{\ell}}{r_{k}}  \right)\\
				\geq &\sum_{j=t}^{r_{i}}\binom{n-s_{\ell}}{r_{k-p}-t} \binom{s_{\ell}}{j} \geq \sum_{j=t}^{r_{i}}\binom{n-s_{\ell}}{r_{i}-t} \binom{s_{\ell}}{j}
				\geq \sum_{j=t}^{r_{i}}\binom{n-s_{\ell}}{r_{i}-j} \binom{s_{\ell}}{j}.
			\end{aligned}
		\end{equation*}  
		This together with (3.8), (3.9) and (3.10) yields
		\begin{equation*}
			\binom{r_{k}}{t}\binom{n-r_{k}}{s_{\ell-p}-t}\geq\sum_{j=t}^{r_{1}}\binom{s_{\ell}}{j}\binom{n-s_{\ell}}{r_{1}-j}+\cdots+\sum_{j=t}^{r_{k-p}}\binom{s_{\ell}}{j}\binom{n-s_{\ell}}{r_{k-p}-j}+1,
		\end{equation*}
		implying that 
		\begin{equation*}
			\begin{aligned} 
				 &\ |Y|-d(X_{k})-\left( |X|-d(Y_{\ell})\right)\\
				=&\ \sum_{j=t}^{s_{1}}\binom{r_{k}}{j}\binom{n-r_{k}}{s_{1}-j}+\cdots+\sum_{j=t}^{s_{\ell-p}}\binom{r_{k}}{j}\binom{n-r_{k}}{s_{\ell-p}-j}\\
				&\ -\left(\sum_{j=t}^{r_{1}}\binom{s_{\ell}}{j}\binom{n-s_{\ell}}{r_{1}-j}+\cdots+\sum_{j=t}^{r_{k-p}}\binom{s_{\ell}}{j}\binom{n-s_{\ell}}{r_{k-p}-j} \right)\\
				\geq &\ \binom{r_{k}}{t}\binom{n-r_{k}}{s_{\ell-p}-t} -\left(\sum_{j=t}^{r_{1}}\binom{s_{\ell}}{j}\binom{n-s_{\ell}}{r_{1}-j}+\cdots+\sum_{j=t}^{r_{k-p}}\binom{s_{\ell}}{j}\binom{n-s_{\ell}}{r_{k-p}-j} \right)>0.     
			\end{aligned}
		\end{equation*}
	The desired result follows.	
	\end{proof}
	\section{Proof of Theorem \ref{2}}	
	Let $n$ and $t$ be positive integers.  Suppose that $R=\left\{ r_{1},   r_{2},   \ldots,   r_{k} \right\}$ and   $S=\left\{  s_{1},  s_{2},  \ldots,  s_{\ell}\right\}$ are subsets of $[n]$ with $r_{1}< r_{2}< \cdots< r_{k}$,   $s_{1}< s_{2}< \cdots< s_{\ell}$, $t\leq \min R\cup S$ and $r_{k}+s_{\ell}-t<n$. 
	Set $X_{i} ={V\brack r_{i}}$, $Y_{j}={V\brack s_{j}}$, $X=X_{1}\cup\cdots \cup X_{k}$ and $Y=Y_{1}\cup\cdots \cup Y_{\ell}.$  
	
	Let $G_{2}(X,Y)$ be a bipartite graph, whose vertex set is $X\biguplus Y$, and two vertices $x\in X$ and $y\in Y$ are adjacent if and only if $\dim(x\cap y)<t$.

	It is easy to verify that $GL(V)$ is an automorphism group of $G_{2}(X, Y)$ stabilizing $X$. Moreover, $X_{1},\ldots, X_{k}$ and   $Y_{1},\ldots ,Y_{\ell}$ are exactly all orbits under the action of $GL(V)$ on $X\cup Y$. 
	
	\begin{prop}\label{7}
		Let $G_{2}(X,  Y)$ be as above. Then the following hold. 
		\begin{enumerate}[\normalfont (i)]
			\vspace{-0.3cm}	\item $G_{2}(X,  Y)$ is a non-complete bipartite graph.
			\vspace{-0.3cm}	\item  $|Y_{j}\backslash N(x_{i})|\leq |Y_{j}\backslash N(x_{i+1})|- 2$ for any $i\in [k-1]$,  $ j\in [\ell]$, $x_{i}\in X_{i}$ and  $x_{i+1}\in X_{i+1}$. Moreover,  $|Y\backslash N(x_{i})|\leq |Y\backslash N(x_{i+1})|-2$.
			\vspace{-0.3cm}	\item  If $r_{i}\geq s_{j}$,   then $|Y_{j}\backslash N(x_{i},  x_{i}^{\prime})|\leq |Y_{j}\backslash N(x_{i})|-1$ for any $x_{i},  
			x_{i}^{\prime}\in X_{i}$.
		\end{enumerate}
	\end{prop}
	
	\begin{proof}[\textnormal{\textbf{Proof.}}] Let $\left\{ \varepsilon_{1},  \varepsilon_{2},  \ldots, \varepsilon_{n} \right\}$ be a basis of $V$.
		
		\textnormal{(i)} Note that $\langle \varepsilon_{1},  \varepsilon_{2},  \ldots, \varepsilon_{r_{1}}\rangle \in X_{1}$ and   $ \langle\varepsilon_{1},  \varepsilon_{2},  \ldots, \varepsilon_{s_{1}}\rangle\in Y_{1}$ are not adjacent.  Hence $G_{2}(X,  Y)$ is not a complete bipartite graph.
		
		\textnormal{(ii)} Since $X_{i}$ and $X_{i+1}$ are two orbits under the action of $GL(V)$, we may assume  $x_{i}=\langle \varepsilon_{1},  \varepsilon_{2},  \ldots, \varepsilon_{r_{i}}\rangle$ and   $x_{i+1}=\langle\varepsilon_{1},  \varepsilon_{2},  \ldots, \varepsilon_{r_{i+1}}\rangle$.   Then $Y_{j}\backslash N(x_{i})\subseteq Y_{j}\backslash N(x_{i+1})$. So, it is sufficient to prove $|(Y_{j}\backslash N(x_{i+1}))\backslash (Y_{j}\backslash N(x_{i}))|\geq 2$.
		Consider the following set
		$$\mathcal{D}\!=\!\left\{ \langle z_{1}\cup z_{2}\cup z_{3}\rangle:\!z_{1}\!\in \!\binom{\left\{\varepsilon_{1},  \varepsilon_{2},  \ldots, \varepsilon_{r_{i}} \right\}}{t-1},   z_{2}\!\in\! {\left\{ \varepsilon_{r_{i}+1},   \ldots, \varepsilon_{r_{i+1}}  \right\}},   z_{3}\!\in\! \binom{\left\{  \varepsilon_{r_{i+1}+1},   \ldots, \varepsilon_{n} \right\}}{s_{j}-t}\right\}\!. $$
		It is easy to check that $\mathcal{D}\subseteq (Y_{j}\backslash N(x_{i+1}))\backslash (Y_{j}\backslash N(x_{i}))$.  Moreover,  if $s_{j}\geq 2$ or $r_{i+1}\geq r_{i}+2$,  then $|\mathcal{D}|\geq 2$.   If $r_{i+1}=r_{i}+1$ and $s_{j}=1$,  then $t=1$. Observe that there are $q^{r_{i}}$ complements of $x_{i}$ in $x_{i+1}$ and   these complements are all in $ (Y_{j}\backslash N(x_{i+1}))\backslash (Y_{j}\backslash N(x_{i}))$. The desired result follows.

		\textnormal{(iii)} Suppose that $\dim(x_{i}\cap x_{i}^{\prime})=m$ and  $\left\{  \alpha_{1},  \ldots , \alpha_{m}   \right\}$ is a basis of $x_{i}\cap x_{i}^{\prime}$. Let $\lbrace   \alpha_{1},  \ldots , \alpha_{m}, $  $ \beta_{1},  \ldots, \beta_{r_{i}-m}\rbrace$ and $\left\{   \alpha_{1},  \ldots , \alpha_{m},  \beta_{1}^{\prime},  \ldots, \beta_{r_{i}-m}^{\prime} \right\}$ be a basis of $x_{i}$ and $x_{i}^{\prime}$,  respectively. Further,  they can be expanded to a basis of $V$ denoted by 
		$\lbrace\alpha_{1},  \ldots , \alpha_{m},  \beta_{1},  \ldots, \beta_{r_{i}-m},   \beta_{1}^{\prime}, $
		$ \ldots, \beta_{r_{i}-m}^{\prime}, $   
		$\eta_{1},  \ldots, \eta_{n-2r_{i}+m}  \rbrace$.
		
		Note that $Y_{j}\backslash N(x_{i},  x_{i}^{\prime}) \subseteq Y_{j}\backslash N(x_{i})$. So, it is sufficient to prove $(Y_{j}\backslash N(x_{i}))  \backslash (Y_{j}\backslash N(x_{i}, x_{i}^{\prime}))\neq \emptyset$. If $m<t$,  then ${x_{i} \brack s_{j}} \subseteq ( Y_{j}\backslash N(x_{i}))\backslash (Y_{j}\backslash N(x_{i},  x_{i}^{\prime}))$. If $m\geq t$ and $r_{i}-m\geq s_{j}-t+1$,  consider the following non-empty set
		$$\mathcal{G}=\left\{ \langle z_{1}\rangle\oplus \langle z_{2}\rangle :z_{1}\in \binom{\left\{ \alpha_{1},  \ldots , \alpha_{m} \right\}}{t-1},    z_{2}\in \binom{\left\{ \beta_{1},  \ldots, \beta_{r_{i}-m}  \right\}}{s_{j}-t+1}\right\}. $$
		It is easy to verify that $\mathcal{G}\subseteq(Y_{j}\backslash N(x_{i}))  \backslash (Y_{j}\backslash N(x_{i}, x_{i}^{\prime}))$. 
		If $m\geq t$ and $r_{i}-m<s_{j}-t+1$,  consider the following non-empty set
		$$\mathcal{H}=\left\{ \langle z_{1} \cup z_{2} \cup \lbrace\beta_{1},  \ldots, \beta_{r_{i}-m}\rbrace\rangle :z_{1}\in \binom{\left\{ \alpha_{1},  \ldots , \alpha_{m} \right\}}{t-1},    z_{2}\in \binom{\left\{ \eta_{1},  \ldots, \eta_{n-2r_{i}+m}  \right\}}{s_{j}+m+1-r_{i}-t}\right\}. $$
		We  easily check that $\mathcal{H}\subseteq(Y_{j}\backslash N(x_{i}))  \backslash (Y_{j}\backslash N(x_{i}, x_{i}^{\prime}))$. 
	\end{proof}
	In order to compute $|Y\backslash N(x)|$ and $|X\backslash N(y)|$,  we need the following well-known result from \cite[Lemma 2.4]{MR2304003} and \cite[Lemma 4]{MR1699558} .
	\begin{lemma}\label{10} \textnormal{(\cite{MR2304003, MR1699558})} 
		Let $n$, $a$, $b$ and $j$ be positive integers with $j\leq \min\left\{ a, b \right\}$ and $a+b-j\leq n$. If $A\in {V \brack a}$,  then $$\left| \left\{ B\in {V \brack b}: \dim(B\cap A)=j \right\} \right|=q^{(a-j)(b-j)}{a \brack j}{n-a \brack b-j}.$$ 
	\end{lemma}

		By Lemma \ref{10}, we have
		$$|Y|-d(X_{k})=\sum_{j=t}^{s_{1}}q^{(r_{k}-j)(s_{1}-j)}{r_{k}\brack j}{n-r_{k}\brack s_{1}-j}+\cdots+\sum_{j=t}^{s_{\ell}}q^{(r_{k}-j)(s_{\ell}-j)}{r_{k}\brack j}{n-r_{k}\brack s_{\ell}-j} , $$
		$$|X|-d(Y_{\ell})=\sum_{j=t}^{r_{1}}q^{(s_{\ell}-j)(r_{1}-j)}{s_{\ell}\brack j}{n-s_{\ell}\brack r_{1}-j}+\cdots+\sum_{j=t}^{r_{k}}q^{(s_{\ell}-j)(r_{k}-j)}{s_{\ell}\brack j}{n-s_{\ell}\brack r_{k}-j}.$$
	So, in order to prove inequality $(1.2)$, we just need to prove $\alpha(X,Y)=\max \lbrace 1+|Y|-d(X_{k}), 1+|X|-d(Y_{\ell})
	\rbrace$.
	\begin{proof}[\textnormal{\textbf{Proof of inequality $(1.2)$.}}]
		By Proposition \ref{7} \textnormal{(ii)} and symmetry,  both $d(X_{1})>d(X_{2})>\cdots>d(X_{k})$ and $d(Y_{1})>d(Y_{2})>\cdots>d(Y_{\ell})$ hold. 
		
		Let $\left\{ \varepsilon_{1},  \varepsilon_{2},  \ldots, \varepsilon_{n} \right\}$ be a basis of $V$. For each $i\in [k]$, in order to prove that $X_{i}$ is not contained in any nontrivial independent set, we just need to check that $N(X_{i})=Y$.  For any $z\in[\ell]$, let $y_{z}=\langle\varepsilon_{1}, \ldots, \varepsilon_{s_{z}}\rangle$,   then $y_{z}$ is adjacent to $$x_{i}=\langle \varepsilon_{1}, \varepsilon_{2}, \ldots, \varepsilon_{t-1},  \varepsilon_{s_{z}+1}, \cdots,  \varepsilon_{s_{z}+r_{i}-(t-1)} \rangle\in X_{i}.$$
		Hence $y_{z}\in N(X_{i})$.   
		For any $y\in Y_{z}$, there exists $\gamma \in GL(V)$ such that $\gamma(y_{z})=y$ since $Y_{z}$ is an orbit under the action of $GL(V)$. Then $y$ is adjacent to $\gamma(x_{i})\in X_{i}$. Consequently $Y_{z}\subseteq N(X_{i})$.  Since $z$ is arbitrary, we have $N(X_{i})=Y$. Similarly,  for each $j\in [\ell]$,  $Y_{j}$ is not contained in any nontrivial independent set.  	
		
		For any $A\in X\cup Y$, it is well known that the stabilizer of $A$ is a maximal subgroup of $GL(V)$\cite{MR0746539}.   Therefore, $GL(V)\mid _{X_{i}}$ and $GL(V)\mid _{Y_{j}}$ are primitive for any $i\in[k]$ and $j\in [\ell]$. Thus all conditions in Theorem \ref{import} hold,  this completes the proof.
	\end{proof} 
	To prove Theorem \ref{2}  $(1)$, we need the following theorem.
		\begin{thm}\textnormal{(\cite{MR2971702})}\label{ref2}		
		Let $n$, $r$, $s$ and $t$ be positive integers with $n\geq 4$, $r,s \geq 2$, $t<\min\lbrace r,s \rbrace$, $r+s-t<n$  and ${n \brack r}\leq {n \brack s}$. If $\mathcal{A}\subseteq {V \brack r}$ and $\mathcal{B}\subseteq {V \brack s}$ are non-empty cross $t$-intersecting, then 
		$$\left|\mathcal{A}\right|+\left|\mathcal{B}\right|\le 	1+{n \brack s}-\sum_{i=0}^{t-1}q^{(r-i)(s-i)}{r \brack i}{n-r \brack s-i}.$$ 
		Moreover equality holds if and only if one of the following holds:
		\begin{enumerate}[\normalfont(i)]
			\vspace{-0.2cm}	\item  $\mathcal{A}=\left\{ A \right\}$ and $\mathcal{B}=\left\{ B\in {V \brack s}: \dim(B\cap A)\geq t\right\}$ for some $A\in {V \brack r}$;
			\vspace{-0.2cm}	\item ${n \brack r}={n \brack s}$ and  $\mathcal{B}=\left\{ B \right\}$ and $\mathcal{A}=\left\{ A\in {V \brack r}: \dim(B\cap A)\geq t \right\}$ for some $B\in {V \brack s}$.
		\end{enumerate}
	\end{thm}
	\begin{proof}[\textnormal{\textbf{Proof of Theorem \ref{2}  $(1)$.}}]
		If $r_{k}=t$,  then Theorem \ref{2} is trivial. 
		Next,  we may assume that $r_{k}>t$. 
		
	    The proof of inequality $(1.2)$   shows that $G_{2}(X, Y)$ satisfies all conditions in Theorem \ref{import}.  Then we have $\mathcal{F}(X,Y)\supseteq\binom{X_{k}}{1}  \cup\binom{Y_{\ell}}{1} \cup \phi \left(\binom{X_{k}}{1}\cup \binom{Y_{\ell}}{1}\right)$ since $|X|-d(Y_{\ell})=|Y|-d(X_{k})$. 
	   Moreover, according to Proposition \ref{7}\textnormal{(ii)} and symmetry,  we get $d(X_{k-1})-1>d(X_{k})$ and $d(Y_{\ell-1})-1>d(Y_{\ell})$. By Proposition \ref{5},  Theorem \ref{import} \textnormal{(iv)} does not hold.   
		
	Suppose for contradiction that  there exists a nontrivial fragment in $\mathcal{F}(X,Y)$.  By Lemma \ref{8}, there exists  non-singleton $F\in \mathcal{F}(X,  Y)$ such that $|F\cap X_{i}|\leq 1$ for any $i\in [k-1]$ and $|F\cap Y_{j}|\leq 1$ for any $j\in [\ell-1]$.  Moreover, we know that $F$ is a minimum-sized nontrivial fragment. By symmetry,   we may assume that $F\in \mathcal{F}(X)$. 	Since $|F|\geq 2$ and Theorem \ref{import} \textnormal{(iv)} does not hold, we have $|F\cap X_{k}|\geq 2$. 
		
		 The definition of $\phi(F)$ implies   $$\bigcup_{C\in \phi(F)}\left\{ B\in {V \brack r_{k}}:  C\subseteq B\right\}\subseteq \phi(F).$$  Thus 
		\begin{equation}\tag{4.1}
			\phi(F)\cap Y_{k}\neq \emptyset. 
		\end{equation}
		
		Note that both $G_{2}(X,  Y)$ and $G_{2}(X_{k},   Y_{k})$ satisfy all conditions  in Theorem \ref{import}.  By Theorem \ref{import} \textnormal{(iii)},   we have, 	for any $z\in X_{k}$,  
		\begin{equation}\tag{$4.2$}
			\alpha(X,  Y)=\alpha(X_{k},  Y_{k})+|Y_{k-1}\backslash N(z)|+\cdots+|Y_{1}\backslash N(z)|.
		\end{equation}
		
		Let $x_{k},   x_{k}^{\prime}\in F \cap X_{k}$. By (4.1) and (4.2), we have
		\begin{equation*}
			\begin{aligned}
				&\ |F\cap X_{k}|+|\phi(F)\cap Y_{k}|+|Y_{k-1}\backslash N(x_{k})|+\cdots+|Y_{1}\backslash N(x_{k})|\\
				\leq &\ \alpha(X_{k},  Y_{k})+|Y_{k-1}\backslash N(x_{k})|+\cdots+|Y_{1}\backslash N(x_{k})|
				=\alpha(X,   Y)=|F|+|\phi(F)|\\
				=&\ |F\cap X_{k}|+|\phi(F)\cap Y_{k}|+|F
				\backslash X_{k}| +|Y_{k-1}\backslash N(F)|+\cdots+|Y_{1}\backslash N(F)|.
			\end{aligned}
		\end{equation*}
		
		 If $F\backslash X_{k}\neq \emptyset$, then $k\geq 2$. 	
		By Proposition \ref{7} \textnormal{(ii)}, similar to  Case 1 in the proof of Theorem \ref{1} $(1)$, we have 
		\begin{equation*}
			\begin{aligned}
				&\ |F\cap X_{k}|+|\phi(F)\cap Y_{k}|+|Y_{k-1}\backslash N(x_{k})|+\cdots+|Y_{1}\backslash N(x_{k})|\\
				\leq &\ |F\cap X_{k}|+|\phi(F)\cap Y_{k}|+|F\backslash X_{k}|+|Y_{k-1}\backslash N(x_{k})|-2|F\backslash X_{k}|+\cdots+|Y_{1}\backslash N(x_{k})|-2|F\backslash X_{k}|,
			\end{aligned}
		\end{equation*}
		which is impossible. 
	If  $F\backslash X_{k}=\emptyset$.
		By Proposition \ref{7} \textnormal{(iii)}, similar to  Case 2 in the proof of Theorem \ref{1} $(1)$, we have 
		\begin{equation*}
			\begin{aligned}
				&\ |F\cap X_{k}|+|\phi(F)\cap Y_{k}|+|Y_{k-1}\backslash N(x_{k})|+\cdots+|Y_{1}\backslash N(x_{k})|\\
				\leq &\ |F\cap X_{k}|+|\phi(F)\cap Y_{k}|+|Y_{k-1}\backslash N(x_{k})|-1+\cdots+|Y_{1}\backslash N(x_{k})|-1\\
				=&\ |F\cap X_{k}|+|\phi(F)\cap Y_{k}|+|Y_{k-1}\backslash N(x_{k})|+\cdots+|Y_{1}\backslash N(x_{k})|-(k-1),
			\end{aligned}
		\end{equation*}
	which implies	$k=1$ and $F$ is a nontrivial fragment in $G_{2}(X_{1}, Y_{1})$. By Theorem \ref{ref2},  this is impossible.
		Therefore,  there are no nontrivial fragments in $G_{2}(X, Y)$.
	\end{proof}
	
	In order to prove Theorem \ref{2}  $(2)$, we need the following  lemma.
	\begin{lemma}\label{11}
		Let $m$ and $i$ be integers with $0\leq i\leq m$.  Then $q^{i(m-i)}\leq {m \brack i}\leq q^{i(m-i+1)}$, and $q^{i(m-i)}< {m \brack i}< q^{i(m-i+1)}$ if $0< i< m$.
	\end{lemma}
	\begin{proof}[\textnormal{\textbf{Proof of Theorem \ref{2}  $(2)$.}}]
		The proof of inequality $(1.2)$  shows that $G_{2}(X, Y)$ satisfies all conditions in Theorem \ref{import}. By Proposition \ref{7}\textnormal{(ii)} and symmetry,  we get $d(X_{k-1})-1>d(X_{k})$ and $d(Y_{\ell-1})-1>d(Y_{\ell})$. According to Proposition \ref{5},  Theorem \ref{import} \textnormal{(iv)} does not hold.
		Then  we just need to prove that $|Y|-d(X_{k})> |X|-d(Y_{\ell})$ when $\max R\bigtriangleup S \in S$ and $|Y|-d(X_{k})< |X|-d(Y_{\ell})$ when $\max R\bigtriangleup S \in R$ by Theorem \ref{import}.  We only prove the former,  the proof of the latter is similar.
		
		Let $s_{\ell-p}=\max R\bigtriangleup S$,  then $r_{k-p+1}=s_{\ell-p+1}$, $r_{k-p+2}=s_{\ell-p+2}$, $\ldots$ , $r_{k}=s_{\ell}$ .  If $k=p$, then $|Y|-d(X_{k})> |X|-d(Y_{\ell})$ holds. 
		Next we may assume $k>p$. 
		
		Note that $s_{\ell -p}> r_{k-p}\geq t$.  Together with Lemma \ref{11} and $r_{k}+s_{\ell-p}-t<n$,  we have 
		\begin{equation*}\tag{$4.3$}
			\begin{aligned}
				q^{(r_{k}-t)(s_{\ell-p}-t)}{r_{k}\brack t}{n-r_{k}\brack s_{\ell-p}-t}&> q^{(r_{k}-t)(s_{\ell-p}-t)+t(r_{k}-t)+(s_{\ell -p}-t)(n-r_{k}-s_{\ell-p}+t)}\\
				&=q^{(s_{\ell-p}-t)n-s_{\ell-p}^{2}+tr_{k}+ts_{\ell-p}-t^{2}}		\geq q^{(s_{\ell-p}-t)n-s_{\ell-p}^{2}+tr_{k}}.
			\end{aligned}
		\end{equation*}
		
	For any  $i\in [k-p]$ and $t\leq j\leq r_{i}$,	by  Lemma \ref{11}, we have
		\begin{equation*}
			\begin{aligned}
				q^{(s_{\ell}-j)(r_{i}-j)}{s_{\ell}\brack j}{n-s_{\ell}\brack r_{i}-j}&\leq q^{(s_{\ell}-j)(r_{i}-j)+j(s_{\ell}-j+1)+(r_{i}-j)(n-s_{\ell}-r_{i}+j+1)}\\
				&=q^{(r_{i}-j)n-r_{i}^{2}+r_{i}+j(r_{i}+s_{\ell}-j)}
				\leq q^{(r_{i}-j)n-r_{i}^{2}+r_{i}+r_{k}(r_{i}+s_{\ell}-t)}\\
				&=q^{(r_{i}-j)n+r_{i}(r_{k}+1-r_{i})+r_{k}s_{\ell}-tr_{k}}
				\leq q^{(r_{i}-j)n+r_{k}^{2}+r_{k}s_{\ell}-tr_{k}},
			\end{aligned}                                             
		\end{equation*}
		implying that
		\begin{equation*}\tag{$4.4$}
			\begin{aligned}
				\sum_{j=t}^{r_{i}}q^{(s_{\ell}-j)(r_{i}-j)}{s_{\ell}\brack j}{n-s_{\ell}\brack r_{i}-j}
				\leq  \sum_{j=t}^{r_{i}} q^{(r_{i}-j)n+r_{k}^{2}+r_{k}s_{\ell}-tr_{k}}
				\leq      q^{(r_{i}-t)n+r_{k}^{2}+r_{k}s_{\ell}-tr_{k}+1}.	\end{aligned}                                           
		\end{equation*}
		
		By computation, we have
		\begin{equation*}
			\begin{aligned} 
				&|Y|-d(X_{k})-\left(|X|-d(Y_{\ell}) \right)\\
				=&\sum_{j=t}^{s_{1}}q^{(r_{k}-j)(s_{1}-j)}{r_{k}\brack j}{n-r_{k}\brack s_{1}-j}+\cdots+\sum_{j=t}^{s_{\ell-p}}q^{(r_{k}-j)(s_{\ell-p}-j)}{r_{k}\brack j}{n-r_{k}\brack s_{\ell-p}-j}\\&-\left(\sum_{j=t}^{r_{1}}q^{(s_{\ell}-j)(r_{1}-j)}{s_{\ell}\brack j}{n-s_{\ell}\brack r_{1}-j}+\cdots+\sum_{j=t}^{r_{k-p}}q^{(s_{\ell}-j)(r_{k-p}-j)}{s_{\ell}\brack j}{n-s_{\ell}\brack r_{k-p}-j} \right).
			\end{aligned}
		\end{equation*}
		This together with (4.3), (4.4) and $s_{\ell-p}>r_{k-p}$ yields
		\begin{equation*}
			\begin{aligned}
				|Y|-d(X_{k})-|X|+d(Y_{\ell})>&\  q^{(s_{\ell-p}-t)n-s_{\ell-p}^{2}+tr_{k}}-q^{(r_{k-p}-t)n+r_{k}^{2}+r_{k}s_{\ell}-tr_{k}+2}\\
				\geq &\  q^{(r_{k-p}-t+1)n-s_{\ell-p}^{2}+tr_{k}}-q^{(r_{k-p}-t)n+r_{k}^{2}+r_{k}s_{\ell}-tr_{k}+2}\\
				\geq &\  q^{(r_{k-p}-t)n-s_{\ell-p}^{2}+tr_{k}+s_{\ell}^{2}+r_{k}^{2}+s_{\ell}r_{k}}-q^{(r_{k-p}-t)n+r_{k}^{2}+r_{k}s_{\ell}-tr_{k}+2}\\
				\geq &\  q^{(r_{k-p}-t)n+r_{k}^{2}+r_{k}s_{\ell}+tr_{k}}-q^{(r_{k-p}-t)n+r_{k}^{2}+r_{k}s_{\ell}-tr_{k}+2}\geq0.
			\end{aligned}                                             
		\end{equation*}
	The desired result follows.
	\end{proof}
	  	

\begin{thebibliography}{00}
		\bibitem{MR1429238} R. Ahlswede and L.H. Khachatrian, The complete intersection theorem for systems of finite sets, European J. Combin. 18 (2) (1997) 125–136. 
		\bibitem{MR0746539} M. Aschbacher, On the maximal subgroups of the finite classical groups, Invent. Math. 76 (3) (1984) 469–514.
		\bibitem{MR4441255} P. Borg and C. Feghali, The maximum sum of sizes of cross-intersecting families of subsets of a set, Discrete Math. 345 (11) (2022) 112981.
		\bibitem{MR0721612} M. Deza and P. Frankl, The Erd\H{o}s-Ko-Rado theorem---{$22$} years later, SIAM J. Algebraic Discrete Methods 4 (4) (1983) 419–431.
		\bibitem{MR0140419} P. Erd\H{o}s, C. Ko and R. Rado, Intersection theorems for systems of finite sets, Quart. J. Math. Oxf. 2 (12) (1961) 313-320.
		\bibitem{MR0519277} P. Frankl, The Erd\H{o}s-Ko-Rado theorem is true for $n=ckt$,  Coll. Math. Soc. J. Bolyai 18 (1978) 365–375.
		\bibitem{MR0810699} P. Frankl and R.L. Graham, Intersection theorems for vector spaces, European J. Combin. 6 (2) (1985) 183–187.
		\bibitem{MR4644282} P. Frankl, E.L. L Liu, J. Wang and Z. Yang, Non-trivial $t$-intersecting separated families, Discrete Appl. Math. 342 (2024) 124–137. 
		\bibitem{MR1178386} P. Frankl and N. Tokushige, Some best possible inequalities concerning cross-intersecting	families, J. Combin. Theory Ser. A  61 (1) (1992) 87–97.
		\bibitem{MR0867648} P. Frankl and R.M. Wilson, The Erd\H{o}s-Ko-Rado theorem for vector spaces, J. Combin. Theory Ser. A 43 (2) (1986) 228–236.
		\bibitem{MR4618236} P. Gupta, Y. Mogge, S. Piga and	B. Sch\"{u}lke, $r$-cross $t$-intersecting families via necessary intersection points, Bull. Lond. Math. Soc. 55 (3) (2023) 1447–1458.
		\bibitem{MR0219428}  A.J.W. Hilton and E.C. Milner, Some intersection theorems for systems of finite sets, Quart. J. Math. Oxf. Ser. 2 (18) (1967) 369–384.
		\bibitem{MR0382016} W.N. Hsieh, Families of intersecting finite vector spaces, J. Combinatorial Theory Ser. A 18 (1975) 252–261.
		\bibitem{MR0382015} W.N. Hsieh, Intersection theorems for systems of finite vector spaces, Discrete Math. 12 (1) (1975) 1–16.
		\bibitem{Liu202330910} E.L. L Liu, The maximum sum of the sizes of cross $t$-intersecting separated families, AIMS Mathematics 8 (12) (2023) 30910 – 30921.
		\bibitem{MR2265507} B. Newton and B. Benesh,  A classification of certain maximal subgroups of symmetric groups, J. Algebra 304 (2) (2006) 1108–1113.
		\bibitem{MR2231096} H. Tanaka, Classification of subsets with minimal width and dual width in	{G}rassmann, bilinear forms and dual polar graphs, J. Combin. Theory Ser. A 113 (5) (2006) 903–910.
		\bibitem{MR2304003} J. Wang and H. Zhang, Normalized matching property of a class of subspace lattices, Taiwanese J. Math. 11 (1) (2007) 43–50.
		\bibitem{MR2971702} J. Wang and H. Zhang, Nontrivial independent sets of bipartite graphs and
		cross-intersecting families, J. Combin. Theory Ser. A 120 (1) (2013) 129–141.
		\bibitem{MR1699558} Y. Wang, On a class of subspace lattices, J. Math. Res. Exposition 19 (2) (1999) 341–348.
		\bibitem{MR0771733} R.M. Wilson, The exact bound in the Erd\H{o}s-Ko-Rado theorem, Combinatorica 4 (1984) 247–257.
	\end{thebibliography}
\end{document}